\documentclass[11pt,a4paper]{article}

\pdfoutput=1

\usepackage[utf8]{inputenc}
\usepackage[T1]{fontenc}
\usepackage[colorlinks]{hyperref}
\usepackage{mathtools, amsthm, amsfonts, amssymb}
\usepackage{accents} 
\usepackage{microtype}

\hypersetup{
  linkcolor=[rgb]{0.3,0.3,0.6},
  citecolor=[rgb]{0.2, 0.6, 0.2},
  urlcolor=[rgb]{0.6, 0.2, 0.2}
}

\usepackage{latexsym}
\mathtoolsset{showonlyrefs,showmanualtags}
\usepackage{thmtools}
\usepackage{mathrsfs}
\usepackage{textcomp}
\usepackage{graphicx}
\usepackage{setspace}
\usepackage{nicefrac}
\usepackage{enumerate}
\usepackage{wasysym}
\usepackage[normalem]{ulem}
\usepackage{upgreek}

\usepackage{paralist}
\usepackage{xcolor}
\usepackage{etoolbox}
\usepackage{mathdots}
\usepackage{wrapfig}
\usepackage{floatflt}
\usepackage{tensor} 
\usepackage{lscape}
\usepackage{stmaryrd}
\usepackage[enableskew]{youngtab}
\usepackage{ytableau}
\usepackage{pifont}

\usepackage{tikz}
\usetikzlibrary{arrows}
\usetikzlibrary{matrix}
\usetikzlibrary{positioning}
\usetikzlibrary{calc}

\usepackage[all,cmtip]{xy}
\usepackage[labelfont=bf, font={small}]{caption}[2005/07/16]

\usepackage{xcolor}
\definecolor{red}{rgb}{1,0,0}

\newcommand{\vvirg}{ , \dots , }
\newcommand{\ootimes}{ \otimes \cdots \otimes }

\newcommand{\ttimes}{ \times \cdots \times }

\newcommand{\wwedge}{ \wedge \dots \wedge}

\newcommand{\textbigotimes}{{\textstyle \bigotimes}}

\newcommand{\calI}{\mathcal{I}}
\newcommand{\calJ}{\mathcal{J}}

\newcommand{\calM}{\mathcal{M}}

\newcommand{\calV}{\mathcal{V}}

\newcommand{\bbC}{\mathbb{C}}

\newcommand{\bbK}{\mathbb{K}}
\newcommand{\bbL}{\mathbb{L}}

\newcommand{\bbN}{\mathbb{N}}

\newcommand{\bbP}{\mathbb{P}}

\newcommand{\bbR}{\mathbb{R}}

\newcommand{\rmQ}{\mathrm{Q}}

\renewcommand{\phi}{\varphi}
\newcommand{\eps}{\varepsilon}

\renewcommand{\tilde}[1]{\widetilde{#1}}
\renewcommand{\hat}[1]{\widehat{#1}}
\renewcommand{\bar}[1]{\overline{#1}}

\newcommand{\id}{\mathrm{id}}

\newcommand{\rank}{\mathrm{rank}}

\DeclareMathOperator{\rk}{rk}
\DeclareMathOperator{\End}{End}

\DeclareMathOperator{\Sym}{Sym}

\newcommand{\Sub}{\mathrm{Sub}}

\DeclareMathOperator{\subrank}{Q}
\DeclareMathAccent{\wtilde}{\mathord}{largesymbols}{"65}
\DeclareMathOperator{\aQ}{\underaccent{\wtilde}{Q}}
\DeclareMathOperator{\asympsubrank}{\underaccent{\wtilde}{Q}}

\newcommand{\GL}{\mathrm{GL}}

\DeclareMathOperator{\Stab}{Stab}

\usepackage[top=2.5cm,bottom=2.5cm,left=3cm,right=3cm]{geometry}

\newcommand{\w}{\mathsf{W}}
\newcommand{\un}{\mathsf{I}}

\usepackage{thmtools, thm-restate} 
\declaretheorem[name=Theorem, parent=section]{theorem}

\declaretheorem[name=Proposition, sibling=theorem]{proposition}
\declaretheorem[name=Lemma, sibling=theorem]{lemma}

\theoremstyle{definition}
\declaretheorem[name=Definition, sibling=theorem]{definition}
\declaretheorem[name=Remark, sibling=theorem]{remark}

\theoremstyle{remark}

\numberwithin{equation}{section}

\newcommand{\im}{\mathrm{im}}
\newcommand{\Gr}{\mathrm{Gr}}

\DeclareMathOperator{\pR}{PR}
\DeclareMathOperator{\sR}{SR}

\newcommand{\degengeq}{\unrhd}

\newcommand{\Cay}{\mathrm{Cay}}

\newcommand{\ext}{{(\eps)}}

\hypersetup{breaklinks=true}

\begin{document}

\begin{center}
	{\LARGE A Gap in the Subrank of Tensors}
\\[1cm] \large

\setlength\tabcolsep{0em}
\newcommand{\myPad}{\hspace{2em}}
\centerline{%
\begin{tabular}{c@{\myPad}c@{\myPad}c}
	Matthias Christandl & Fulvio Gesmundo & Jeroen Zuiddam\\[0.2em]
    \textsf{christandl@math.ku.dk} & \textsf{fgesmund@math.univ-toulouse.fr} & \textsf{j.zuiddam@uva.nl}\\[0.2em]
     University of Copenhagen &  Saarland University  & University of Amsterdam
\end{tabular}%
}

\vspace{9mm}

\large

\vspace{9mm}
\bf Abstract
\end{center}
\normalsize
\noindent
The subrank of tensors is a measure of how much a tensor can be ``diagonalized''. This parameter was introduced by Strassen to study fast matrix multiplication algorithms in algebraic complexity theory and is closely related to many central tensor parameters (e.g.~slice rank, partition rank, analytic rank, geometric rank, G-stable rank) and problems in combinatorics, computer science and quantum information theory. 
Strassen (J.~Reine Angew.\ Math., 1988) proved that there is a gap in the subrank when taking large powers under the tensor product: either the subrank of all powers is at most one, or it grows as a power of a constant strictly larger than one. In this paper, we precisely determine this constant for tensors of any order. Additionally, for tensors of order three, we prove that there is a second gap in the possible rates of growth. Our results strengthen the recent work of Costa and Dalai (J.~Comb.\ Theory, Ser.~A, 2021), who proved a similar gap for the slice rank. Our theorem on the subrank has wider applications by implying such gaps not only for the slice rank, but for any ``normalized monotone''. In order to prove the main result, we characterize when a tensor has a very structured tensor (the W-tensor) in its orbit closure. Our methods include degenerations in Grassmanians, which may be of independent interest.
\thispagestyle{empty}

\renewcommand{\footnotesize}{\normalsize}
\let\thefootnote\relax
\let\thefootnote\relax\footnotetext{2020 Math. Subj. Class.: (primary) 15A69, (secondary) 14N07, 15A72, 68R05}


\section{Introduction} \label{intro}

We prove a structural theorem about ``diagonalizing'' tensors and in particular about a tensor parameter called \emph{subrank} \cite{Str:RelativeBilComplMatMult}. This parameter was originally introduced to study fast matrix multiplication algorithms, and is closely related to many recently introduced tensor parameters, such as slice rank~\cite{tao}, analytic rank \cite{MR2773103,MR3964143}, geometric rank \cite{KopMosZui:GeomRankSubrankMaMu} and G-stable rank~\cite{MR4471036}, and to a variety of problems in combinatorics, computer science and quantum information theory. Our results improve on recent results of Costa and Dalai~\cite{DBLP:journals/jcta/CostaD21} on gaps in the slice rank and on bounds of Strassen from 1988 \cite{strassen1988asymptotic} on gaps in the subrank.

Informally, the subrank of a $k$-tensor measures how much the tensor can be ``diagonalized'' by taking linear combinations of its slices, in any of the $k$ directions (\autoref{subsec: defin subrank and powers}). For matrices, namely in the case $k=2$, the subrank coincides with the matrix rank, and in particular it is very well understood and easy to compute. On the other hand, for tensors of higher order, that is for $k\geq 3$, much is still unknown about the subrank. 

Motivated by various applications of tensor methods to problems with a recursive structure (for instance, the cap set problem in combinatorics \cite{MR3583358}) we are interested in the behaviour of the subrank when taking large Kronecker powers of a tensor. The notion of Kronecker product that we use is the natural generalization of the Kronecker product of two matrices: the product of two tensors is a tensor of the same order whose entries are all the pair-wise products of the entries of the two factors (\autoref{subsec: defin subrank and powers}).
In this paper, we show the following:
\begin{itemize}
    \item We prove a gap in the subrank of tensors when taking large powers: for every nonzero $k$-tensor $T$ over any field, either the subrank of every power of $T$ is $1$, or the subrank of the $N$-th power of $T$ is at least $(k/(k-1)^{(k-1)/k})^{N - o(N)}$ for every $n$.
    \item We prove a second gap for $3$-tensors: for every nonzero $3$-tensor $T$ over any field, either the subrank of every power of $T$ is $1$, or the subrank of the $N$-th power of $T$ grows as 
    \[
    (3/2^{2/3})^{N - o(N)} \approx (1.88\ldots)^{N-o(N)},
    \]or the subrank of the $N$-th power of $T$ is at least $2^N$ for every~$N$.
    \item We prove, as a consequence of the above results, that any ``normalized monotone'' has a gap as above. In particular this applies to slice rank and partition rank, but also other tensor parameters such as analytic rank, geometric rank and G-stable rank, thus extending a recent result of Costa and Dalai~\cite{DBLP:journals/jcta/CostaD21}.
    \item As a key ingredient for the proof of the above, we give a sufficient and necessary condition for any tensor to have a very structured tensor, namely the W-tensor of order $k$, in its orbit-closure.
    \item In the course of proving these results, we prove several properties of the partition rank and an equivalence between degeneration of tensors and degenerations of elements in certain Grassmanians.
\end{itemize}

Our results are based on methods from algebraic geometry and invariant theory, and in particular the study of degenerations of tensors, degenerations of subspaces of tensors (in Grassmannians), and orbit classifications. In the rest of the introduction we will provide more details on our results.

\subsection{Subrank and tensor powers}\label{subsec: defin subrank and powers}

Before discussing our main results, we set some basic notation, define the subrank and define the notion of tensor power that we use. 

Throughout this introduction, we let $\bbK$ be an arbitrary field (unless otherwise specified). We use the integer $k \geq 2$ to denote the order of our tensors. Let $\bbK^{n_1} \ootimes \bbK^{n_k}$ denote the space of tensors of order $k$ (i.e., $k$-tensors) with coefficients in~$\bbK$ and with dimensions $n_1, \ldots, n_k \in \bbN$. We let $e_1, e_2, \ldots, e_{n_i}$ denote the standard basis vectors in $\bbK^{n_i}$.

The subrank of a $k$-tensor $T \in \bbK^{n_1} \ootimes \bbK^{n_k}$, denoted by $\rmQ(T)$, is the largest integer~$r$ such that there are linear maps $\pi_i : \bbK^{n_i} \to \bbK^{r}$ with the property that 
\[
(\pi_1 \ootimes \pi_k)T = \sum_{i=1}^r e_i \ootimes e_i
\]
Intuitively, the tensor $T$ is ``diagonalized'' by taking linear combinations of the slices of $T$ according to the maps $\pi_i$. 
When $k=2$ the subrank $\subrank(T)$ coincides with the rank of $T$ as a matrix. (For $k\geq3$, however, the subrank does not coincide with the well-known tensor rank!) If $T = 0$ then the subrank is zero and if $T$ is not zero then the subrank is at least one. We will thus focus on non-zero tensors. The subrank of a tensor in $\bbK^{n_1} \ootimes \bbK^{n_k}$ is at most $\min_i n_i$.

For two $k$-tensors $T \in \bbK^{n_1} \ootimes \bbK^{n_k}$ and $S \in \bbK^{m_1} \ootimes \bbK^{m_k}$ their Kronecker product $T \boxtimes S \in \bbK^{n_1m_1} \ootimes \bbK^{n_km_k}$ is the $k$-tensor obtained by taking $T \otimes S \in \bbK^{n_1} \ootimes \bbK^{n_k} \otimes \bbK^{m_1} \ootimes \bbK^{m_k}$, regrouping the tensor factors into $(\bbK^{n_1} \otimes \bbK^{m_1}) \ootimes (\bbK^{n_k} \otimes \bbK^{m_k})$, and identifying $\bbK^{n_i} \otimes \bbK^{m_i}$ with $\bbK^{n_im_i}$. In other words, thinking of $T$ and $S$ as $k$-dimensional arrays of elements of $\bbK$, their tensor product $T \boxtimes S$ is the $k$-dimensional array whose coefficients are the pairwise products of the coefficients of $T$ and the coefficients of $S$. Having defined the product~$\boxtimes$ we can naturally take the $n$th power $T^{\boxtimes N} \in \bbK^{n_1^N} \ootimes \bbK^{n_k^N}$.

We will be interested in how the subrank $\subrank(T^{\boxtimes N})$ grows as $N$ grows, and we ask the natural question: What ``rates of growth'' are possible? We will prove that the possible rates of growth are more ``rigid'' than what one might a priori expect.

\subsection{A gap in the subrank of tensors}
Our main result is the following gap in the subrank of powers of tensors. Recall that $\bbK$ is an arbitrary field.

\begin{restatable}[Subrank gap]{theorem}{thmsubrankgap}\label{th:subrank-gap}
For every nonzero $T \in \bbK^{n_1} \ootimes \bbK^{n_k}$, one of the following is true: \begin{enumerate}[\upshape(a)]
    \item $\subrank(T^{\boxtimes N}) = 1$ for all $N$;
    \item $\subrank(T^{\boxtimes N}) \geq c_k^{N - o(N)}$ for all $N$, where $c_k = k /(k-1)^{(k-1)/k}$.
\end{enumerate}
\end{restatable}

So \autoref{th:subrank-gap} says that if we have a strict inequality $\subrank(T^{\boxtimes N}) > 1$ for \emph{any} $N$, then we can asymptotically ``boost'' this inequality to $\subrank(T^{\boxtimes N}) \geq c_k^{N - o(N)}$ for all $N$. 

The constant $c_k$ coincides with $2^{h(1/k)}$ where $h$ is the binary entropy function, defined by $h(p) = -p \log_2 p -{(1-p)}\log_2 {(1-p)}$ for $p \in (0,1)$, and $h(0) = h(1) = 0$. This constant appears in \autoref{th:subrank-gap} because it controls the rate of growth of the subrank of a special tensor called the W-tensor, denoted by $\w_k$: it is known \cite{strassen1991degeneration} that $\subrank(\w_k^{\boxtimes N}) = c_k^{N - o(N)}$. In particular, this fact implies that the constant $c_k$ in \autoref{th:subrank-gap} is optimal. For small values of $k$ we have $c_2 = 2$, $c_3 \approx 1.88988$, $c_4 \approx 1.75477$, and $c_5 \approx 1.64938$. For all $k$ we have $c_k > 1$, and $c_k$ is decreasing and converges to $1$ as $k$ diverges to infinity.

\subsection{A gap for partition rank and other normalized monotones}
With the same methods that we use to prove \autoref{th:subrank-gap}, we obtain gaps for other tensor parameters, and in particular we strengthen the recent result of Costa and Dalai \cite{DBLP:journals/jcta/CostaD21}. Costa and Dalai study a tensor parameter called slice rank, which was defined by Tao \cite{tao} and which we will denote by $\sR(T)$. They prove the following gap theorem:

\begin{theorem}[Slice rank gap \cite{DBLP:journals/jcta/CostaD21}]
 For every nonzero $T \in \bbK^{n_1} \ootimes \bbK^{n_k}$ exactly one of the following is true:
 \begin{enumerate}[\upshape(a)]
     \item $\sR(T^{\boxtimes N}) = 1$ for all $N$;
     \item $\sR(T^{\boxtimes N}) \geq c_k^{N - o(N)}$ for all $N$.
 \end{enumerate}
\end{theorem}

As a consequence of \autoref{th:subrank-gap}, we extend the above result to partition rank. The partition rank was defined by Naslund~\cite{DBLP:journals/jcta/Naslund20a} as a natural variation on the slice rank. Denote by $\pR(T)$ the partition rank of the tensor $T$; see \autoref{sec:basic} for the precise definition. Partition rank is at most slice rank. We prove:

\begin{restatable}[Partition rank gap]{theorem}{thmprgap}\label{th:pr-gap} 
For every nonzero $T \in \bbK^{n_1} \ootimes \bbK^{n_k}$ exactly one of the following is true:
 \begin{enumerate}[\upshape(a)]
     \item $\pR(T^{\boxtimes N}) = 1$ for all $N$;
     \item $\pR(T^{\boxtimes N}) \geq c_k^{N - o(N)}$ for all $N$.
 \end{enumerate}
\end{restatable}

In fact, we obtain the above kind of gap for a general class of tensors parameters that includes the slice rank, partition rank, geometric rank, normalized analytic rank, G-stable rank and subrank. 
To discuss this class of tensor parameters we need some concepts that we will now discuss. 

For two tensors $T \in \bbK^{n_1} \ootimes \bbK^{n_k}$ and $S \in \bbK^{m_1} \ootimes \bbK^{m_k}$ we let $T \geq S$ and say ``$T$ restricts to $S$\,'' if there are linear maps~$\pi_i : \bbK^{n_i} \to \bbK^{m_i}$ such that $(\pi_1 \ootimes \pi_k) T = S$. Moreover, for any $r \in \bbN$,  define $\un_{k,r} = \sum_{i=1}^r e_i \ootimes e_i \in \bbK^{r} \ootimes \bbK^{r}$: the tensor $\un_{k,r}$ is often called the ``identity tensor'' or the ``unit tensor'' of order $k$ and rank $r$. Note that the subrank~$\subrank(T)$ is the largest number $r$ such that $T \geq \un_{k,r}$.

The class of tensor parameters we consider is as follows. Let $f$ be a function from the set of tensors of order $k$ to $\bbR_{\geq 0}$. We call $f$ a \emph{normalized monotone} if $f(S) \leq f(T)$ whenever $S$ is a restriction of $T$, and $f(\un_{k,r}) = r$ for all $r \in \bbN$. It follows directly that $f(T) \geq \subrank(T)$ for any tensor $T$.

For any $k$-tensor $T \in \bbK^{n_1} \ootimes \bbK^{n_k}$ the \emph{flattenings} $T_I$ of $T$ are the $2$-tensors obtained by grouping the factors $\bbK^{n_i}$ into two groups: $T_I \in (\textbigotimes_{i \in I} \bbK^{n_i}) \otimes (\textbigotimes_{j \notin I} \bbK^{n_j})$. We prove:
\begin{restatable}[Gap for normalized monotones]{theorem}{thmgeneralgap}\label{th:f-gap}
 Let $f$ be any normalized monotone. For every~$T \in \bbK^{n_1} \ootimes \bbK^{n_k}$, if there is no flattening of $T$ of rank one, then $f(T^{\boxtimes N}) \geq c_k^{N- o(N)}$ for all $N$.
\end{restatable}

\begin{remark}\label{rem:norm-mon}
The slice rank and partition rank (and many other parameters) have the property that they are bounded from above by the ranks of the flattenings. For any normalized monotone~$f$ that is bounded from above by the flattening ranks we have that for every nonzero tensor $T \in \bbK^{n_1} \ootimes \bbK^{n_k}$ exactly one of the following is true:
\begin{enumerate}[\upshape(a)]
    \item $f(T^{\boxtimes N}) = 1$ for all $N$;
    \item $f(T^{\boxtimes N}) \geq c_k^{N- o(N)}$ for all $N$.
\end{enumerate}
\end{remark}

\subsection{A second gap in the subrank for tensors of order three}

For the special case of tensors of order three, we prove a stronger version of \autoref{thm: degeneration to w}.
Recall that \autoref{thm: degeneration to w} states that for any tensor $T$ of order $k$ there are two possibilities for the rate of growth of the subrank under taking large tensor powers: either $\subrank(T^{\boxtimes N}) \leq 1$ for all~$N$, or $\subrank(T^{\boxtimes N}) \geq c_k^{N - o(N)}$ for all $N$, where $c_k = k/(k-1)^{k/(k-1)}$. In other words, the subrank of $T^{\boxtimes N}$ either stays at most $1$ for all $N$, or it grows at least like~$c_k^{N - o(N)}$, so that there is a gap in the possible rates of growth.

For tensors of order three, we prove that there is a \emph{second} gap in the possible rates of growth of the subrank: if the rate of growth is strictly larger than $c_3 = 3/2^{3/2} \approx 1.88$, then it is at least $2$. More precisely:

\begin{restatable}{theorem}{thmasympsubtrich}
\label{th:asympsubtrich}
For every nonzero $T \in \bbK^{n_1} \otimes \bbK^{n_2} \otimes \bbK^{n_3}$ exactly one of the following is true:
\begin{enumerate}[\upshape(a)]
    \item $\subrank(T^{\boxtimes N}) = 1$ for all $N$;
    \item $\subrank(T^{\boxtimes N}) = c_3^{N - o(N)}$ for all $N$, where $c_3 = 3/2^{2/3} \approx 1.88$;
    \item $\subrank(T^{\boxtimes N}) \geq 2^{N-o(N)}$ for all $N$.
\end{enumerate}
\end{restatable}

So \autoref{th:asympsubtrich} not only tells us that if $\subrank(T^{\boxtimes N}) > 1$ for some $N$, then $\subrank(T^{\boxtimes N}) \geq c_3^{N - o(N)}$ for all $N$, but also that if $\subrank(T^{\boxtimes N}) \geq d^{N - o(N)}$ for some constant $d > c_3$, then $d \geq 2$. In fact, if~$\bbK$ is algebraically closed, we can prove that the lower bound in case (c) of \autoref{th:asympsubtrich} is~$2^N$ rather than $2^{N - o(N)}$.

\subsection{Our gaps in terms of asymptotic subrank}
The result of \autoref{th:subrank-gap} and the other results of this section can be phrased uniformly in terms of the asymptotic subrank of a tensor, another tensor parameter introduced by Strassen \cite{strassen1988asymptotic}, see also \cite{cvz}.
The \emph{asymptotic subrank} of $T \in \bbK^{n_1} \ootimes \bbK^{n_k}$ is the limit $\aQ(T) \coloneqq \lim_{N\to\infty} \subrank(T^{\boxtimes N})^{1/N}$ and thus describes the asymptotic rate of growth of the subrank when taking large powers of $T$. This limit exists by a result called Fekete's Lemma, and can be replaced by a supremum over $N \in \bbN$. The asymptotic subrank a priori can take any real value in the closed interval $[1,\min_i ( n_i)]$. It is easy to see that $\aQ(\un_{k,r}) = r$. \autoref{th:subrank-gap} can be phrased in terms of asymptotic subrank as follows (by directly applying the definition of asymptotic subrank):

\begin{restatable}{theorem}{thmasymsubrank}\label{th:asysubrank}
For every nonzero $T \in \bbK^{n_1} \ootimes \bbK^{n_k}$, one of the following is true: 
\begin{enumerate}[\upshape(a)]
    \item $\aQ(T) = 1$;
    \item $\aQ(T) \geq c_k$ where $c_k = k /(k-1)^{(k-1)/k}$.
\end{enumerate}
\end{restatable}

Similarly, \autoref{th:asympsubtrich} can be phrased as follows using the definition of asymptotic subrank:

\begin{theorem}\label{thm: asy rank 3 tensors}
For every nonzero $T \in \bbK^{n_1} \otimes \bbK^{n_2} \otimes \bbK^{n_3}$, exactly one of the following is true: 
\begin{enumerate}[\upshape(a)]
    \item $\aQ(T) = 1$;
    \item $\aQ(T) = c_3$ where $c_3 = 3 /2^{2/3} \approx 1.88$;
    \item $\aQ(T) \geq 2$.
\end{enumerate}
\end{theorem}

\subsection{Degeneration to the W-tensor}\label{sec: intro degeneration}

We now discuss our approach to proving \autoref{th:subrank-gap}. 
The main ingredient is a structural result on tensors that is of independent interest. This structural result characterizes which tensors admit a very structured tensor, the W-tensor, as a degeneration, in the sense explained below. The characterization that we will present here is in terms of a simple criterion based on the flattening ranks of the tensor.

We briefly introduce the necessary notions. In this part we let $\bbK$ be an algebraically closed field. For any dimensions $n_1, \ldots, n_k \geq 2$, let $\w_k \in \bbK^{n_1} \ootimes \bbK^{n_k}$ be the tensor defined as follows. Let $e_1, e_2, \ldots$ be the standard basis vectors in $\bbK^{n_i}$ and define $\w_k = \sum_s e_{s_1} \ootimes e_{s_k}$ where the sum is over all $k$-tuples $s \in \{1,2\}^k$ that are permutations of $(2, 1, \ldots, 1)$. In other words, $\w_k$ is the tensor with coefficients in $\{0,1\}$ and support given by the $k$-tuples $(2, 1, \ldots, 1), (1, 2, \ldots, 1), \ldots, (1, 1, \ldots, 2)$. For example,
\begin{align*}
\w_2 &= e_2 \otimes e_1 + e_1 \otimes e_2\\
\w_3 &= e_2 \otimes e_1 \otimes e_1 + e_1 \otimes e_2 \otimes e_1 + e_1 \otimes e_1 \otimes e_2\\
\w_4 &= e_2 \otimes e_1 \otimes e_1 \otimes e_1 + e_1 \otimes e_2 \otimes e_1 \otimes e_1 + e_1 \otimes e_1 \otimes e_2 \otimes e_1 + e_1 \otimes e_1 \otimes e_1 \otimes e_2. 
\end{align*}

The notion of degeneration is the approximate version of restriction. Intuitively $T$ degenerates to $S$ if there are arbitrary small perturbations of $S$ to which $T$ restricts. Precisely, given $T,S \in \bbK^{n_1} \ootimes \bbK^{n_k}$, $T$ degenerates to $S$, denoted $T \degengeq S$, if $S$ is in the Zariski closure of the orbit of $T$, that is $S \in \bar{\{ (g_1, \ldots, g_k) T : g_i \in \GL_{n_i}\}}$. When $\bbK = \bbC$, then the closure can be equivalently taken in the Euclidean topology. More generally, one can define an equivalent notion of degeneration which mimics the behavior of limits in the Euclidean topology; we refer to \autoref{sec:degen-grass} for further explanations. Similarly to restriction, degeneration is an ordering on tensors; it is ``weaker'' than restriction, in the sense that, if $T \geq S$, then $T \degengeq S$, and there are examples for which the reverse implication does not hold. 

Two natural questions are the following: for which tensors $T \in \bbK^{n_1} \ootimes \bbK^{n_k}$ is there a restriction $T \geq \w_k$ and for which is there a degeneration $T \degengeq \w_k$? We solve the second question by providing a sufficient and necessary condition. This condition is in terms of so-called flattening ranks of the tensor, which we briefly mentioned before. Given a tensor $T$, let $T_I \in (\bigotimes_{i \in I} \bbK^{n_i}) \otimes (\bigotimes_{i \not\in I} \bbK^{n_i})$ denote the $I$-flattening of $T$. It is not hard to see that all flattening ranks of $\w_k$ are $2$. This fact, together with semicontinuity of matrix rank, implies that if $T \degengeq \w_k$, then all flattening ranks of~$T$ are at least $2$. We prove that this necessary condition is also sufficient:

\begin{restatable}{theorem}{thmdegentoW}\label{thm: degeneration to w}
For every $T \in \bbK^{n_1} \ootimes \bbK^{n_k}$ exactly one of the following is true:
\begin{enumerate}[\upshape(a)]
    \item $T$ has a flattening of rank one;
    \item $T \degengeq \w_k$. 
\end{enumerate}
\end{restatable}
For tensors of order three, \autoref{thm: degeneration to w} can be obtained via a reduction to the case $n_1 = n_2 = n_3 = 2$ and the classical characterization of the orbits in $\bbK^2 \otimes \bbK^2 \otimes \bbK^2$, dating back essentially to \cite{Sylv:PrinciplesCalculusForms}. In fact, for $k=3$ we prove a more precise classification with three cases. This classification not only involves the W-tensor $\w_3 = e_2 \otimes e_1 \otimes e_1 + e_1 \otimes e_2 \otimes e_1 + e_1 \otimes e_1 \otimes e_2$ but also the unit tensor $\un_{3,2} = e_1 \otimes e_1 \otimes e_1 + e_2 \otimes e_2 \otimes e_2$ of order three and rank two. We prove the following.
\begin{restatable}{theorem}{thmorderthreecomplex}\label{thm: subrank 2 for 3tensors}
For every nonzero $T \in \bbK^{n_1} \otimes \bbK^{n_2} \otimes \bbK^{n_3}$ exactly one of the following is true:
\begin{enumerate}[\upshape(a)]
    \item $T$ has a flattening of rank one;
    \item $\w_3 \geq T$ and $T \geq \w_3$;
    \item $T \geq \un_{3,2}$.
\end{enumerate}
\end{restatable}

Condition (b) in \autoref{thm: subrank 2 for 3tensors} says that there are linear maps $\pi_i, \sigma_i : \bbK^{n_i} \to \bbK^{n_i}$ such that $(\pi_1 \otimes \pi_2 \otimes \pi_3) T = \w_3$ and $(\sigma_1 \otimes \sigma_2 \otimes \sigma_3) \w_3 = T$. This condition is equivalent to the statement that $T \in \bbK^{n_1} \otimes \bbK^{n_2} \otimes \bbK^{n_3}$ is isomorphic to $\w_3 \in \bbK^{n_1} \otimes \bbK^{n_2} \otimes \bbK^{n_3}$, that is, there are invertible linear maps $\pi_i : \bbK^{n_i} \to \bbK^{n_i}$ such that $(\pi_1 \otimes \pi_2 \otimes \pi_3) T = \w_3$.

As a sanity check, notice that \autoref{thm: subrank 2 for 3tensors} implies \autoref{thm: degeneration to w} for tensors of order three. Indeed, it is known that $\un_{3,2} \degengeq \w_3$, that is the border rank of $\w_3$ is $2$, so if $T$ does not satisfy condition (a) of \autoref{thm: subrank 2 for 3tensors} then $T \degengeq \w_3$.

\subsection{Related work}

\paragraph*{Gaps for slice rank.} 
\autoref{th:subrank-gap} extends a recent result of Costa and Dalai~\cite{DBLP:journals/jcta/CostaD21}, which says that for any $k$-tensor~$T$, either the asymptotic slice rank of $T$ is at most $1$ or it is at least $c_k = 2^{h(1/k)}$ (the same constant as in our~\autoref{th:subrank-gap}). It is pointed out in~\cite{DBLP:journals/jcta/CostaD21} how this gap provides a barrier for the slice rank to give good upper bounds on certain combinatorial problems (which we discuss more in \autoref{subsec:app}). Their proof relies on combinatorial methods to study the slice rank of powers of a tensor that were introduced by Tao and Sawin~\cite{sawin}. 

Since slice rank is at least subrank, the result of \cite{DBLP:journals/jcta/CostaD21} follows from \autoref{th:asysubrank}. We moreover obtain the same gap for any ``normalized monotone'' (\autoref{rem:norm-mon}) for instance partition rank, analytic rank, geometric rank and G-stable rank. Whether the asymptotic subrank can be strictly smaller than the asymptotic slice rank remains an open problem. It is known, however, that the subrank is significantly smaller than slice rank for generic (i.e.~``most'') tensors~\cite{derksen_et_al:LIPIcs.CCC.2022.9}.

\vspace{-.4cm}\paragraph*{Previous bounds for asymptotic subrank.} \autoref{th:subrank-gap} improves on results of Strassen~\cite{strassen1988asymptotic}. Namely, \cite[Lemma~3.7]{strassen1988asymptotic} says that if $T$ is a tensor having no flattening of rank one, then $\aQ(T) \geq 2^{2/k}$. (The proof is given in the case of $3$-tensors but it generalizes directly to tensors of any order.) Since $c_k = 2^{h(1/k)} > 2/k$ for every~$k \geq 3$, \autoref{th:subrank-gap} improves on the bound of \cite{strassen1988asymptotic}. 
    
Tensors satisfying the property of being ``balanced'' are known to satisfy a stronger lower bound on the asymptotic subrank \cite[Proposition~3.6]{strassen1988asymptotic}. In particular, since generic tensors are balanced, this result guarantees that if $T$ is a generic $k$-tensor in $(\bbK^n)^{\otimes k}$, then $\aQ(T) \geq n^{2/k}$. 

\vspace{-.4cm}\paragraph*{Values of normalized monotones over finite fields are well-ordered.} 
In the recent \cite[Corollary~1.4.3]{BlaDraRup:TensorRestrictionFiniteFields}, the authors prove that for any real-valued normalized monotone $f$ on $k$-tensors over a finite field $\bbK$, the image  of $f$ is a well-ordered set, that is, any subset of the image of $f$ has a smallest element. Equivalently, every strictly decreasing sequence of elements of the image of $f$ terminates after finitely many steps. This guarantees that for any normalized monotone $f$, for every $r > 0$ there is a $\delta > 0$ such that no tensor $T$ satisfies $r < f(T) < r+\delta$. In  \autoref{th:asysubrank}, we explicitly determine the largest such $\delta$ for $r = 1$ to be $\delta = c_k - 1$. In \autoref{thm: asy rank 3 tensors}, we determine the largest such $\delta$ for $r = c_3$ to be $\delta = 2 - c_3$.

\vspace{-.4cm}\paragraph*{Homomorphism duality.} \autoref{thm: degeneration to w} and \autoref{thm: subrank 2 for 3tensors} say that \emph{non-existence} of certain degenerations is equivalent to \emph{existence} of other degenerations. More precisely, \autoref{thm: degeneration to w} states that for any $k$-tensor $T$, there is no degeneration $S \degengeq T$ for any $k$-tensor $S$ having a flattening of rank one if and only if there is a degeneration $T \degengeq \w_k$. This property is related to the notion of ``forbidden restrictions'' studied in \cite{BlaDraRup:TensorRestrictionFiniteFields}. This phenomenon is generally known as ``homomorphism duality'' and was introduced in graph theory to study ``gaps'' in the homomorphism ordering~\cite[Section 1.4]{DBLP:books/daglib/0013017}. An example of a homomorphism duality in graph theory is the following: given any graph $G$, there is no homomorphism $G \to K_1$ if and only if there is a homomorphism $K_2 \to G$, where $K_n$ denotes the complete graph on $n$ vertices. Another example is the theorem of K\"onig stating that there is no homomorphism $G \to K_2$ if and only if there is a homomorphism $C_\ell \to G$ for some odd integer $\ell \geq 3$, where~$C_\ell$ denotes the cycle graph on $\ell$ vertices.

\subsection{Applications of our results}\label{subsec:app}

\paragraph*{Boosts and barriers for tensor methods in discrete mathematics.}
The breakthrough results of Croot, Lev and Pach \cite{MR3583357} and Ellenberg and Gijswijt \cite{MR3583358} on the cap set problem lead to much interest in the use of tensor methods to solve problems in combinatorics. Two important tensor parameters in this context are the following: the slice rank \cite{tao, sawin}, introduced to give a simplified solution to the cap set problem, and used to study the sunflower problem \cite{MR3668469}, and the group-theoretic approach to fast matrix multiplication \cite{MR3631613}); the partition rank \cite{DBLP:journals/jcta/Naslund20a}, introduced to solve a combinatorial problem on corners. In a typical application of these tensor methods to bound the size of a combinatorial object, one designs a fixed small tensor $T$ that ``encodes'' the object of study in such a way that bounding from above the slice rank or the partition rank of tensor powers of $T$ gives an upper bound on the combinatorial problem of interest. Our gap results say that the set of values that these upper bounds can have has gaps. Depending on the application such gaps can either pose a barrier for tensor methods to give interesting bounds, or ``boost'' a bound obtained by such a method to the next possible value that the tensor method can take, and thus improve the bound on the combinatorial problem.

\vspace{-.4cm}\paragraph*{Discreteness of quantum entanglement distillation rates.}
In the context of quantum information theory, \autoref{thm: degeneration to w} and \autoref{thm: subrank 2 for 3tensors} can be interpreted as follows. Whenever a quantum state $\Psi$ is genuinely multiparty entangled (which is equivalent to having non-trivial partition rank), then \autoref{thm: degeneration to w} guarantees that it is possible to distill a W-state to arbitrary accuracy by stochastic local operations and classical communication (SLOCC). Furthermore, a rate of $h(1/k)$ Greenberger–Horne–Zeilinger states can be extracted via SLOCC from many copies of $\Psi$. This is a remarkable statement on the possibility of entanglement distillation, especially since the property of genuinely multiparty entanglement can be tested efficiently with the product test~\cite{harrow2013testing}. Furthermore, in the context of certifying genuine multiparty entanglement with help of entanglement polytopes, as discussed in~\cite{WalDorGroChr:EntanglementPolytopes}, we point out that our results imply that any non-trivial entanglement polytope necessarily contains the entanglement polytope of the W-state, a result that should be seen in the light of the quantitative results for qubits in \cite[Suppl.~Material, Lemma~5]{WalDorGroChr:EntanglementPolytopes}. 

\vspace{-.4cm}\paragraph*{Comon-type separations of degeneration and symmetric degeneration.}
In the geometric setting, one often considers the subspace of symmetric tensors $S^k \bbK^n$ of $(\bbK^n)^{\otimes k}$. In this case, one considers \emph{symmetric} restrictions and degenerations, where the underlying group action is the one of the diagonal $\GL_n \subseteq \GL_n^{\times n}$ acting simultaneously on all factors. A long-standing problem posed by Comon \cite[Problem 15]{oedingaim} asked whether the tensor rank of a symmetric tensor coincides with its symmetric tensor rank; in other words, the problem asks whether the existence of a restriction from a unit tensor to $T$ implies the existence of a symmetric restriction. An example where this is indeed not possible was provided in \cite{Shitov:CounterexampleComon}. The analogous problem for subrank was posed in \cite{ChrFawTaZui:SymmetricSubrank} and answered in \cite{Shitov:SubrankVsSymSubrank}. For degenerations, \cite{Chang:MaximalBorderSubrank} provides ways to construct tensors $T \in S^k \bbK^n$ admitting a degeneration to $\un_{k,n}$, whereas it is easy to see that this cannot be achieved via symmetric degeneration. The degeneration analog of Comon's original question on tensor rank instead remains open. However, in \cite{ChrFawTaZui:SymmetricSubrank}, it is shown that in the asymptotic setting the notions of restriction, degeneration and their symmetric versions are all equivalent. \autoref{thm: degeneration to w} provides further examples in the degeneration setting. For example, let $\lambda = (\lambda_1 \vvirg \lambda_s)$ be a partition of $k$ and let $D_\lambda$ be the corresponding Dicke state; under the natural identification between symmetric tensors and homogeneous polynomials, we have $D_\lambda = x_1^{\lambda_1} \cdots x_s^{\lambda_s}$; under this identification $\w_k = x_1^{k-1}x_2$. In particular, if $\lambda_j \neq 1$ for every $j$, then there is no symmetric degeneration from $D_\lambda$ to $\w_k$, because $D_\lambda$ has no factors of multiplicity one; on the other hand \autoref{thm: degeneration to w} guarantees that $D_\lambda \degengeq \w_k$.

\section{Restriction, unit tensor, subrank, flattening, partition rank}\label{sec:basic}

In this section we discuss basic tensor concepts that we will need throughout the paper.

Let $\bbK$ be an algebraically closed field. Let $V_1 \vvirg V_k, W_1 \vvirg W_k$ be finite-dimensional vector spaces over $\bbK$. 
\begin{definition}
Let $T \in V_1 \otimes \cdots \otimes V_k$ and $S \in W_1 \otimes \cdots \otimes W_k$.  We say that $T$ \emph{restricts to} $S$, and write $T \geq S$, if there are linear maps $\pi_i : V_i \to W_i$ such that $(\pi_1 \otimes \cdots \otimes \pi_k) T = S$. 
\end{definition}

\begin{definition}
For $k,r \in \bbN$  the unit tensor of order $k$ and rank $r$ is $\un_{k,r} \in \bbK^{r} \ootimes \bbK^{r}$ defined by
\[
\un_{k,r} = \sum_{i=1}^r e_{i} \ootimes e_i.
\]
\end{definition}

\begin{definition}
Let $T \in V_1 \ootimes V_k$.
The \emph{subrank} of $T$, denoted by $\subrank(T)$, is the largest number $r$ such that $T \geq \un_{k,r}$.
\end{definition}

\begin{definition}
Let $T \in V_1 \ootimes V_k$.
Every subset $I \subseteq [k]$ defines a linear map $T_I : \textbigotimes_{i \in I} V_i^* \to \textbigotimes_{j \notin I} V_j$
via tensor contraction. We call $T_I$ the \emph{$I$-flattening} of $T$. For any $p \in [k]$ we will use the notation $T_p \coloneqq T_{\{p\}}$. We call $T_p : V_p^* \to \bigotimes_{j\neq p} V_j$ the \emph{$p$-flattening} of $T$.
\end{definition}

\begin{definition}
Let $T \in V_1 \ootimes V_k$. We say that $T$ has \emph{partition rank one} if there exists a subset $I \subseteq [k]$ with $I \neq \emptyset$ and $I \neq [k]$, such that $\rk(T_I) = 1$. The \emph{partition rank} of $T$ is the smallest number $r$ such that we can write $T = S_1 + \cdots + S_r$ for tensors $S_i$ that each have partition rank 1.
\end{definition}

This paper is concerned with the structural difference between tensors having partition rank one and tensors having partition rank strictly larger thank one. In particular, in \autoref{sec:prank-restr}, we will provide a characterization of this property. A basic property of the partition rank is the following.
\begin{lemma}[Naslund \cite{DBLP:journals/jcta/Naslund20a}]\label{lem:pr-norm}
For every $r \in \bbN$, we have $\pR(\un_{k,r}) = r$.
\end{lemma}
Another basic property of the partition rank is that it is monotone under restriction:
\begin{lemma}\label{lem:pr-mon}
If $T \geq S$ then $\pR(T) \geq \pR(S)$.
\end{lemma}
\begin{proof}
This is an immediate consequence of the fact that matrix rank is monotone under restriction. Explicitly, let $S = (\pi_1 \ootimes \pi_k)(T)$ and let $T = T_1 + \cdots + T_r$ be an expression of $T$ as a sum of tensors $T_j$ with $r = \pR(T)$ and $\pR(T_j) =1$. By linearity 
\[
 S = (\pi_1 \ootimes \pi_k)(T_1 + \cdots + T_r) = (\pi_1 \ootimes \pi_k)(T_1)  + \cdots + (\pi_1 \ootimes \pi_k)(T_r);
\]
clearly, for every $j$, $(\pi_1 \ootimes \pi_k)(T_j)$ is either $0$ or a tensor of partition rank one, because if a certain flattening of $T_j$ has rank one, then the same flattening of $(\pi_1 \ootimes \pi_k)(T_j)$ has rank at most one. Hence, $S$ admits an expression as sum of $r$ tensors having partition rank one, which guarantees $\pR(S) \leq \pR(T)$.
\end{proof}

It follows directly from \autoref{lem:pr-norm}, \autoref{lem:pr-mon} and the definition of subrank that partition rank bounds subrank from above:
\begin{lemma}\label{lem:subrank}
For every tensor $T$ we have $\subrank(T) \leq \pR(T)$.
\end{lemma}

\section{Preserving non-trivial partition rank under restriction}\label{sec:prank-restr}

In this section we will prove that the property of a tensor having partition rank at least two is preserved under at least one restriction, and therefore under almost all restrictions (i.e.~generic restrictions).

\begin{proposition}\label{prop: good flattenings under restriction}
 Let $k \geq 3$.
 Let $T \in V_1 \ootimes V_k$ have partition rank at least two. 
 For every~$i \in [k]$ let~$W_i$ be a vector space with $\dim(W_i) \geq 2$.
 
 \begin{enumerate}[\upshape(i)]
     \item There are linear maps $\pi_i : V_i \to W_i$ such that $(\pi_1 \ootimes \pi_k)T$ has partition rank at least two.
     \item Let $\pi_i : V_i \to W_i$ be generic linear maps. Then $(\pi_1 \ootimes \pi_k)T$ has partition rank at least two.
 \end{enumerate}
\end{proposition}

The meaning of the term ``generic'' in Claim (ii) of \autoref{prop: good flattenings under restriction} is that there is a nonempty Zariski open subset $U$ of the set of all $k$-tuples of linear maps~$V_i \to W_i$ such that for every $(\pi_1, \ldots, \pi_k)\in U$ we have that $(\pi_1 \ootimes \pi_k)T$ has partition rank at least two. Claim (ii) in \autoref{prop: good flattenings under restriction} follows from a standard semicontinuity argument from Claim (i), since the property of having partition rank at least two is a Zariski-open condition.

\begin{remark}
We do not know whether an analog of \autoref{prop: good flattenings under restriction} holds for higher partition rank. In particular, we do not know whether there are tensors $T$ with $\pR(T) = r$ but with the property that $\pR(T') < r$ for all restrictions $T'$ of $T$ to $\bbK^r \ootimes \bbK^r$. In fact, our proof of \autoref{prop: good flattenings under restriction} relies on \autoref{lemma: characterization good flattening ranks}, which relies on the fact that $\pR(T) \geq 2$ is detected by flattenings ranks; this is not true for higher partition rank, and it is unclear whether some analog of \autoref{lemma: characterization good flattening ranks} can hold in general. 

We point out, however, that there exists an integer $m_{k,r}$, depending only on $r$ and on the number of tensor factors $k$, such that a tensor $T \in V_1 \ootimes V_k$ satisfies $\pR(T) \leq r$ if and only if all restrictions $T'$ of $T$ to $\bbK^{m_{k,r}} \ootimes \bbK^{m_{k,r}}$ satisfy $\pR(T') \leq r$. In particular, $m_{k,r}$ does not depend on the dimensions of $V_1 \vvirg V_k$. This is a consequence of the theory of polynomial functors; we only briefly outline a sketch of the proof, and we refer to  \cite{Draisma:TopologicalNoetherianityPolyFunctors,Bik:StrengthNoethThesis} for the theory. Consider the tensor product functor $\otimes$ sending the tuple of vector spaces $(V_1 \vvirg V_k)$ to $V_1 \ootimes V_k$. For every $r$, the assignment $X_r(V_1 \vvirg V_k) = \bar{\{ T \in V_1 \ootimes V_k : \pR(T) \leq r\}}$ defines a \emph{closed subset} of the functor $\otimes$, in the sense of \cite[Def. 1.3.18]{Bik:StrengthNoethThesis}. By Noetherianity, see \cite[Corollary 9]{Draisma:TopologicalNoetherianityPolyFunctors}, $X_r$ is defined by finitely many $(\GL_\infty^{\times k})$-modules of equations; define $m_{k,r}$ to be the smallest integer such that all these modules appear in $\Sym(\bbK^{m_{k,r}} \ootimes \bbK^{m_{k,r}})$. Then the desired property is satisfied. The result of  \autoref{prop: good flattenings under restriction} is that $m_{k,1} = 2$. In \cite[Prop. 3.1]{karam:highminorshightensors}, the author provides an example of a tensor $T \in \bbK^4 \otimes \bbK^{11} \otimes \bbK^{15}$ with $\pR(T) = 4$ and the property that every ``coordinate'' restriction of $T$ to $\bbK^4 \otimes \bbK^4 \otimes \bbK^4$ has partition rank at most $3$. If one could prove this holds for all restrictions, then one would obtain $m_{k,3} > 4$.
\end{remark}

The proof of \autoref{prop: good flattenings under restriction} uses two lemmas. The first one is a simple classical fact about linear subspaces of rank-one matrices. We will give the proof for the convenience of the reader.

\begin{lemma}\label{lemma: rank one subspaces}
 If $U \subseteq V_1 \otimes V_2$ is a linear subspace such that all elements of $U$ have rank one, then either $U \subseteq v_1 \otimes V_2$ for some element $v_1 \in V_1$ or $U \subseteq V_1 \otimes v_2$ for some element $v_2 \in V_2$.
\end{lemma}
\begin{proof}
 Suppose to the contrary that there exist elements $T = v_1 \otimes v_2 \in U$ and $T' = v_1' \otimes v_2' \in U$ such that the vectors $v_1$ and $v_1'$ are linearly independent, and the vectors $v_2$ and $v_2'$ are linearly independent. Then $T + T'$ has rank two and is an element of $U$, which gives a contradiction.
\end{proof}

The second lemma characterizes a tensor having partition rank at least two in a recursive fashion in terms of the image of the flattenings of the tensor. For tensors of order three, the same result follows essentially from \cite[Proposition 22]{Geng:GeometricRankLinDetVars}. Recall that for any $p \in [k]$ we use the notation $T_p = T_{\{p\}}$ for the $p$-flattening of $T$.

\begin{lemma}\label{lemma: characterization good flattening ranks}
 Let $k\geq 3$ and $T \in V_1 \ootimes V_k$. The following are equivalent:
 \begin{enumerate}[\upshape(a)]
     \item The partition rank of $T$ is at least two.
     \item For every $p \in [k]$ we have
     \begin{enumerate}[\upshape(i)]
     \item $\rk(T_p) \geq 2$ and
     \item $\im(T_p)$ contains an element with partition rank at least two.
     \end{enumerate}
     \item For some $p \in [k]$ we have \begin{enumerate}[\upshape(i)]
     \item $\rk(T_p) \geq 2$ and
     \item $\im(T_p)$ contains an element with partition rank at least two.
     \end{enumerate}
 \end{enumerate}
\end{lemma}
\begin{proof}
We first prove $(a) \Rightarrow (b)$.
We give the proof for $p = k$. For the other values of $p$ the claim follows from the same proof by reordering the factors. Suppose that $\pR(T) \geq 2$. Clearly condition (i) holds. To show that condition (ii) holds, we need to show that there exists $S \in \im (T_k) \subseteq V_1 \ootimes V_{k-1}$ such that, for every subset $J \subseteq [k-1]$ with $J \neq \emptyset$ and $J\neq [k-1]$ we have $\rk(S_J) \geq 2$. 
 Note that for every fixed $J$, the condition $\rk(S_J) \geq 2$ is Zariski-open. The intersection of any two non-empty Zariski-open subsets is a non-empty Zariski-open subset. Therefore, it suffices to show that for every~$J$ there exists $S \in \im (T_k)$ such that $\rk(S_J) \geq 2$. Fix $J$ and suppose for a contradiction that for every $S \in \im (T_k)$ we have $\rk(S_J) =1$. 
 Then $\im ( T_k)$ is a subspace of rank-one elements in $( \bigotimes _{j \in J} V_j ) \otimes (\bigotimes _{i \not\in J} V_i)$. 
 Then by \autoref{lemma: rank one subspaces} we have that $\im (T_k) \subseteq a_J \otimes (\bigotimes _{i \not\in J} V_i)$ for some vector $a_J \in \bigotimes _{j \in J} V_j$ or $\im ( T_k) \subseteq ( \bigotimes _{j \in J} V_j ) \otimes b_{J}$ for some vector $b_{J} \in \bigotimes _{i \not\in J} V_i$. 
 In the first case we find $T \in a_J \otimes (\bigotimes _{i \not\in J} V_i) \otimes V_k$. This implies $\rk(T_J) = 1$ which is in contradiction with the assumption $\pR(T) \geq 2$. We similarly obtain a contradiction in the second case.

 The implication (b) $\Rightarrow$ (c) is clear. 
 
 We now prove (c) $\Rightarrow$ (a). By condition (ii), there exists an $S \in \im (T_k)$ with $\pR(S) \geq 2$. By condition (i) there exists an $S' \in \im (T_k)$ that is linearly independent from $S$. Then, there exist vectors $v^{(k)}_1,v^{(k)}_2 \in V_k$ and a subspace $V_k' \subseteq V_k$ such that the intersection of the linear span $\langle  v^{(k)}_1,v^{(k)}_2 \rangle$ and $V_k'$ is zero, and 
 \[
  T = S \otimes v^{(k)}_1 + S' \otimes v^{(k)}_2 + T'
 \]
for some $T' \in V_1 \ootimes V_{k-1} \otimes V_k'$. Let $\tilde{T} = S \otimes v^{(k)}_1 + S' \otimes v^{(k)}_2 $. The tensor $\tilde{T}$ is the image of the linear map $V_k \to V_k/V_k'$ applied to the $k$th factor.
Since the rank of flattenings are non-increasing under restriction, it suffices to show that $\rk(\tilde{T}_I) \geq 2$ for every subset $I \subseteq [k]$ with $I \neq \emptyset$ and $I \neq [k]$. We distinguish three cases:
 \begin{itemize}
 \item Suppose $I = \{k \}$. The image of $\tilde{T}_k$ is the linear span $\langle S,S'\rangle$. Since $S$ and $S'$ are linearly independent, $\rk( \tilde{T}_k) = 2$. 
 \item Suppose  $|I|\geq 2$ and $I$ contains $k$. We further restrict $\tilde T$ by applying the linear map $V_k/V_k' \to V_k / (V_k' + \langle v^{(k)}_{2} \rangle)$ on the $k$th factor, obtaining $\hat{T} = S \otimes v^{(k)}_1 $. Then $\rk( \tilde{T}_I) \geq \rk( \hat{T}_{I})$. Then $\im  ( \hat{T}_{I}) = \im ( {S}_{I \setminus \{k\}})$. Since $\pR(S) \geq 2$, we obtain $\rk(\hat{T}_{I}) \geq 2$. 
 \item Suppose $I$ does not contain $k$. 
 Consider the transpose flattening, obtained by replacing~$I$ with its complement $I^c = [k] \setminus I$. Since $\rk(\tilde{T}_I) = \rk(\tilde{T}_{I^c})$, we reduce to one of the previous two cases. 
\end{itemize}
We conclude that $\pR(\tilde{T}) \geq 2$, and therefore $\pR(T) \geq 2$.
\end{proof}

\begin{proof}[Proof of \autoref{prop: good flattenings under restriction}]
Let $T_1 : V_1^* \to V_2 \ootimes V_k$ be the first flattening of $T$. Since $\pR(T) \geq 2$, \autoref{lemma: characterization good flattening ranks} guarantees that $\rank(T_1) \geq 2$ and that $\im (T_1) $ contains an element~$S$ with $\pR(S) \geq 2$.  Let $v_1,v_2 \in V_1^*$ be such that $T_1(v_1) = S$ and such that $T_1(v_2)$ is linearly independent from $S$. For any subset $A \in V^*$ let $A^\perp = \{v \in V : \forall f \in A, f(v) = 0\}$ be the annihilator of $A$.
Let $W_1 = V_1 / \langle v_1,v_2 \rangle^\perp$ and let $\pi_1$ be the quotient projection $V_1 \to W_1$. We let 
\[
T^{(1)} = (\pi_1 \otimes \id \otimes \cdots \otimes \id) T \in W_1 \otimes V_2 \ootimes V_k.
\]
Then $\pR(T^{(1)}) \geq 2$, since by construction the flattening $T^{(1)}_1$ satisfies the two conditions (i) and (ii) of \autoref{lemma: characterization good flattening ranks} (c). For every $j = 2 \vvirg k$ we let $T^{(j)}$ be the tensor obtained from $T^{(j-1)}$ by applying the argument above on the $j$-th flattening $T^{(j-1)}_j$ of $T^{(j-1)}$. 
Then we obtain tensors
\[
T^{(j)} \in W_1 \ootimes W_{j}\otimes V_{j+1} \ootimes V_k
\]
that are restrictions of $T$ and that satisfy $\pR(T^{(j)}) \geq 2$. In particular, $T^{(k)} \in W_1 \otimes \cdots \otimes W_k$ is a restriction of $T$ and satisfies $\pR(T^{(k)}) \geq 2$. Note that $\dim(W_i) = 2$.
This shows that there exist linear maps $\pi_j : V_j \to W_j$ with $\dim(W_i) = 2$ such that $\pR((\pi_1 \ootimes \pi_k)T) \geq 2$. Claim~(i) follows immediately. Since the statement of Claim~(i) is open in the Zariski topology, the same property holds for generic linear maps, which gives Claim (ii).
\end{proof}

\section{Degenerations and Grassmannians}\label{sec:degen-grass}

In this section, we characterize tensor degenerations in terms of degeneration of the image of the flattening, regarded as elements of a certain Grassmannian. In order to state this result precisely, we introduce some additional notions. Further, it is useful to work with an equivalent notion of degeneration, that we introduce in \autoref{def: algebraic degeneration}. We discuss the equivalence between this notion of degeneration and the one given in \autoref{intro}, then we state and prove the main result of this section, \autoref{thm: degenerations via grassmannians}. 

The notion of algebraic degeneration mimics the operation of taking a limit along a curve in the Euclidean topology. To define this precisely, we introduce some notation. Let $G = \GL(V_1) \ttimes \GL(V_k)$ be a product of general linear groups acting linearly on a vector space $V$. The action gives a polynomial map $G \to \GL(V) \subseteq \End(V)$. Let $\bbK[[\eps]]$ denote the ring of power series in one variable $\eps$ and let $\bbK((\eps))$ denote the quotient field of $\bbK[[\eps]]$. Define $V^{\ext} = V \otimes \bbK((\eps))$, and regard it as a $\bbK((\eps))$-vector space. Let $G^\ext = \GL(V_1^\ext) \ttimes \GL(V_1^\ext)$. The action of $G$ on $V$ extends by linearity to an action of $G^\ext$ on $V^\ext$.
\begin{definition}[Algebraic degeneration]\label{def: algebraic degeneration}
 Let $G$ be a product of general linear groups acting linearly on $V$ and let $T,S \in V$. We say that \emph{$S$ is an algebraic $G$-degeneration of $T$} if there exist elements $g_\eps \in G^{\ext}$ and $U_\eps \in V^{\ext}$ such that 
 \[
  g_\eps \cdot T = S + U_\eps
 \]
where $\frac{1}{\eps}U_\eps \in V \otimes \bbK[[\eps]]$. In this case, we will often write $  g_\eps \cdot T = S + O(\eps)$.
\end{definition}
Informally, one can think of the group element $g_\eps$ in \autoref{def: algebraic degeneration} as a ``curve'' of elements in $G$ (parametrized by $\eps$), so that $g_\eps \cdot T$ is a curve in $V$ with the property that $S$ belongs to its closure. Identifying $G$ with a group of invertible matrices, one can think of the elements of~$G^\ext$ as invertible matrices whose entries are power series in the variable $\eps$. 

\autoref{def: algebraic degeneration} is seemingly different from the definition of degenerations of tensors given in \autoref{sec: intro degeneration} via closure in the Zariski topology. It is a classical fact however that these two definitions are equivalent. In other words, $S$ is an algebraic degeneration of $T$, in the sense of \autoref{def: algebraic degeneration} if and only if $S \in \bar{G \cdot T}$. The proof of this fact goes back to Hilbert \cite{Hilbert:UberVollenInvariantensysteme}. For an arbitrary algebraically closed field, the statement is proved in \cite[Sec. 20.6]{BuClSho:Alg_compl_theory} for the action of $\GL(V_1) \times \GL(V_2) \times \GL(V_3)$ on $V_1 \otimes V_2 \otimes V_3$ and in \cite[III.2.3, Lemma 1]{Kra:Geom_Meth_Invarianterntheorie} in the case of $\GL(V)$. The proof in these two special settings is essentially the same, and it applies to the general setting: $S$ is an algebraic $G$-degeneration of $T$ (as in \autoref{def: algebraic degeneration}) if and only if $S$ is in the $G$-orbit closure of $T$.

\begin{remark}\label{rmk: epsilons in and out of product}
Given two spaces $V_1,V_2$, it is a standard fact that $V_1^{(\eps)} \otimes_{\bbK((\eps))} V_2^{(\eps)}$ is isomorphic as a $\bbK((\eps))$-vector space to $(V_1 \otimes_\bbK V_2)^{(\eps)}$. We will omit the subscript from the notation of tensor product and write simply $V_1 \otimes V_2$ or $V_1^{(\eps)} \otimes V_2^{(\eps)}$ meaning $\otimes_\bbK$ in the first case and $\otimes_{\bbK((\eps))}$ in the second case. Similarly, if $V$ is a $\bbK$-vector space, then $\Lambda^k V^{(\eps)}$ is to be read as an exterior power with respect to $\otimes_{\bbK((\eps))}$ whereas $\Lambda ^k V$ is an exterior power with respect to $\otimes_\bbK$.
\end{remark}

An important case that we will consider is the one of the action of a group on a Grassmannian. Given a vector space $W$ and an integer $r$, let $\Gr(r,W)$ be the Grassmannian of $r$-planes in~$W$. Then $\Gr(r,W)$ is a projective variety in $\bbP \Lambda^r W$ via the Pl\"ucker embedding. In this way, a group $G$ acting linearly on $W$, acts on $\Gr(r,W)$ via its induced action on $\Lambda^r W$. \autoref{thm: degenerations via grassmannians} characterizes certain tensor degenerations in terms of degeneration of elements of $\Gr(r,W)$ via this induced action. As a special case of \autoref{thm: degenerations via grassmannians}, let $W = \bbK^{r} \ootimes \bbK^r$, $G = \GL_r \ttimes \GL_r$, and $A = \bbK^r$. Let $T = \un_{k,r}$ be the unit tensor in $A \otimes W$. Then \autoref{thm: degenerations via grassmannians} characterizes the border rank of $S$. This characterization was obtained already in \cite[Theorem~2.5]{BucLan:RanksTensorsAndGeneralization} and \cite[Lemma~2.4]{GesOneVen:PartiallySymComon}.

\begin{theorem}\label{thm: degenerations via grassmannians}
 Let $G$ be a product of general linear groups acting linearly on a space $W$. Let $A$ be a vector space and let $r = \dim A$. Let $T,S \in A \otimes W$ be two elements such that the flattenings $T,S: A^* \to W$ are injective and let $E_T,E_S$ be their images regarded as elements of $\Gr(r,W)$. The following are equivalent:
 \begin{enumerate}[\upshape(a)]
  \item $S$ is a $(\GL(A) \times G)$-degeneration of $T$;
  \item $E_S$ is a $G$-degeneration of $E_T$. 
 \end{enumerate}
\end{theorem}

\begin{proof}
By \autoref{rmk: epsilons in and out of product}, when considering the augmentation of the action of $\GL(A) \times G $ on $A \otimes W$, we can consider the group $\GL(A^{\ext}) \times G^{\ext}$.

Let $a_1 \vvirg a_r$ be a $\bbK$-basis of $A$, which is also regarded as a $\bbK((\eps))$-basis of $A^{\ext}$ and let $\alpha_1 \vvirg \alpha_r$ be the dual basis of $A^*$, also regarded as a basis of ${A^{\ext}}^*$.

First we prove that (a) implies (b). Suppose that $S$ is a ($\GL(A) \times G$)-degeneration of $T$. By definition of algebraic degeneration, there exists an element $(x_\eps,g_\eps) \in (\GL(A)\times G)^{\ext} \simeq \GL(A^{\ext}) \times G^{\ext}$ such that 
\[
 (x_\eps \otimes g_\eps)\cdot T = S + O(\eps).
\]
Denote $(x_\eps \otimes g_\eps) \cdot T$ by $T_\eps$. Then $T_\eps \in A^{\ext} \otimes W^{\ext}$. The image $E_{T_{\eps}}$ of the flattening map $T_\eps : A^{[\eps] *} \to W^{\ext}$ has dimension $r$ and has basis
\[
\{ g_\eps T( x_\eps \cdot  \alpha_1) \vvirg g_\eps T( x_\eps \cdot  \alpha_r) \}.
\]
Regarding $E_{T_\eps}$ as an element of $\Lambda^r V^{\ext}$, we have  
\begin{align*}
E_{T_\eps} &= g_\eps \cdot T( x_\eps \cdot  \alpha_1) \wwedge g_\eps \cdot T( x_\eps \cdot  \alpha_r)\\ &= 
(S + O(\eps))(\alpha_1) \wwedge (S + O(\eps))(\alpha_r)\\ &= S(\alpha_1)  \wwedge S(\alpha_r) + O(\eps).
\end{align*}
where the tensor $T$ is identified with its flattening map $T: A^* \to W$.

Notice that, by assumption, the flattening $S : A^* \to W$ is injective and so $E_S$ has dimension $r = \dim(A)$. Therefore, $S(\alpha_1)  \wwedge S(\alpha_r)$ is nonzero and equals $E_S$. 

We obtain that $E_{T_\eps} = E_S + O(\eps)$ is an element of $\Lambda^r W \otimes \bbK[[\eps]]$. We have $E_{(x_\eps \otimes g_\eps) \cdot T} = E_{g_\eps \cdot T}$ since the action of $x_\eps \in \GL(A^{\ext})$ does not change the subspace. Therefore
\[
E_S + O(\eps) = E_{T_{\eps}} = E_{(x_\eps \otimes g_\eps) \cdot T} = E_{g_\eps \cdot T} = g_\eps \cdot E_T.
\]
Hence $E_S$ is a $G$-degeneration of $E_T$. 

Now we prove that (b) implies (a). Assume $E_S$ is a $G$-degeneration of $E_T$. By definition, there exists $g_\eps \in G^{\ext}$ such that $g_\eps \cdot E_T = E_S + O(\eps)$ as an element of $\Lambda^r W^{\ext}$. In particular, $g_\eps \cdot E_T$ is an element of $\Lambda^k W \otimes \bbK[[\eps]]$.

On the one hand,
\[
E_S + O(\eps) =  g_\eps \cdot E_T = g_\eps \cdot (T (\alpha_1)  \wwedge T (\alpha_r)) = (g_\eps \cdot (T_1(\alpha_1))) \wwedge (g_\eps \cdot (T(\alpha_r))) \in \Lambda^rW^{\ext}.
\]
On the other hand, since $E_S + O(\eps)$ is an element of $\Lambda^k W \otimes \bbK[[\eps]]$, we have 
\[
E_S + O(\eps) = S_{1,\eps} \wwedge S_{r,\eps}
\]
for elements $S_{j,\eps} \in W \otimes \bbK[[\eps]]$ such that $S_{j,\eps} = S(\alpha_j) + O(\eps)$.

Define $S_\eps = \sum a_i \otimes S_{j,\eps} \in A^{\ext} \otimes W^{\ext}$ and notice $S_\eps = S + O(\eps)$. Moreover, the flattening image of $S_\eps$ is the same as the one of $g_\eps \cdot T$. This guarantees that there is an element $x_\eps \in \GL(A^{\ext})$ such that $x_\eps (g_\eps \cdot T) = S_\eps$. This results in $(x_\eps \otimes g_\eps) (T) = S_\eps = S + O(\eps)$ as elements of $A^{\ext} \otimes W^{\ext}$, showing that $S$ is a $(\GL(A) \times G)$-degeneration of $T$. This concludes the proof.
\end{proof}

We conclude with an observation providing a generalization of \autoref{thm: degenerations via grassmannians} to the action of arbitrary algebraic linear groups $G$. This requires some familiarity with the language of schemes and properties of discrete valuation rings and is independent from the rest of this work.
\begin{remark}
One can define a notion of algebraic degeneration for any algebraic linear group~$G$, rather than only in the case where $G$ is a product of general linear groups. This definition is given by considering $G$ as a group scheme defined over $\bbK$, and $G^{\ext} = G(\bbK((\eps))$ the set of its $\bbK((\eps))$-points.

In this way, given an algebraic group $G$ acting linearly on a space $V$, one has two natural notions of degenerations: one topological, given as set of points in the closure of a $G$-orbit in the Zariski topology, and one algebraic, given as in \autoref{def: algebraic degeneration} using the action of $G^{\ext}$ on~$V^{\ext}$. These two notions are equivalent in the general setting as well, and this allows one to prove \autoref{thm: degenerations via grassmannians} in the more general setting. However, the proof of the equivalence is more delicate and the arguments of \cite[Sec. 20.6]{BuClSho:Alg_compl_theory} and \cite[III.2.3, Lemma 1]{Kra:Geom_Meth_Invarianterntheorie} require some modifications involving the geometry of curves over discrete valuation rings.
\end{remark}

\section{Degeneration to the W-tensor}
\label{sec:degen-to-W}

In this section, we will prove that if a $k$-tensor has partition rank at least two, then it degenerates to the tensor $\w_k$. We recall the definition of the tensor $\w_k$. For every $j \in [k]$ let~$V_j$ be a vector space of dimension at least two and for every $j \in [k]$ we let $v^{(j)}_0, v^{(j)}_1 \in V_j$ be any two linearly independent vectors. The tensor $\w_k \in V_1 \otimes \cdots \otimes V_k$ is defined as
\[
 \w_k = v^{(1)}_1 \otimes v^{(2)}_0 \ootimes v^{(k)}_0 + \cdots + v^{(1)}_0 \ootimes v^{(k-1)}_0 \otimes v^{(k)}_1.
\]
Up to isomorphism of tensors, this tensor does not depend on the choice of vectors $v_0^{(j)}, v_1^{(j)}$.

\begin{theorem}\label{thm: degeneration to w explicit}
 Let $T \in V_1 \ootimes V_k$ be a tensor with partition rank at least two. Then $\w_k$ is in the orbit closure of $T$, that is, $\w_k \in \bar{(\GL(V_1) \ttimes \GL(V_k)) \cdot T}$.
\end{theorem}

\autoref{thm: degeneration to w explicit} directly implies \autoref{thm: degeneration to w}, stated in the introduction. We prove \autoref{prop: degen with Stabw} as preparation to proving \autoref{thm: degeneration to w explicit}.

\begin{proposition}\label{prop: degen with Stabw}
Let $V_1, \ldots, V_k$ be vector spaces with $\dim V_j = 2$ for every $j \in [k]$. Let $P \in V_1 \ootimes V_k$ be any tensor that is linearly independent from the tensor $\w_{k}$.
Let $H$ be the stabilizer subgroup of $\w_{k}$ under the action of $\GL(V_1) \ttimes \GL(V_k)$ on $V_1 \ootimes V_k$. Then $v^{(1)}_0  \ootimes v^{(k)}_0 \in \bar{( \bbK^\times \times H ) \cdot P} $, where $\bbK^\times = \bbK\setminus \{0\}$ acts by scalar multiplication.
\end{proposition}
\begin{proof}
Write $v_{i_1 \vvirg i_k} = v^{(1)}_{i_1}  \ootimes v^{(k)}_{i_k}$. The elements $\{v_{i_1, \ldots, i_k} : i_1, \ldots, i_k = 0,1\}$ form a basis of $V_1 \ootimes V_k$. Write $P$ as a linear combination of the $v_{i_1, \ldots, i_k}$ with coefficients $P_{i_1, \ldots, i_k} \in \bbK$
\[
 P = \sum_{i_1 \vvirg i_k} P_{i_1 \vvirg i_k} v_{i_1 \vvirg i_k}
\]
where the sum goes over all $i_1, \ldots, i_k \in \{0,1\}$.
First, assume $P_{0\vvirg 0} \neq 0$. Define the matrix
\[
 h_\eps = \left( \begin{array}{cc} \eps^{-1} & 0\\0 & \eps^{k-1} \end{array} \right).
\]
Then $h_\eps^{\otimes k} \in H$. Let $g_\eps = \eps^{k} \cdot h_\eps^{\otimes k}$. Then
\[
 g_\eps  v_{i_1 \vvirg i_k}  = \eps ^{k(i_1 + \cdots + i_k) } v_{i_1 \vvirg i_k}.
\]
In particular, $g_\eps P = P_{0\vvirg 0}v_{0 \vvirg 0} + O(\eps)$. This guarantees $v_{0 \vvirg 0} \in \bar{(\bbK^\times \times H ) \cdot P}$.

Suppose that $P_{0 \vvirg 0} = 0$. 
 Define
 \begin{align*}
 h(s_1 \vvirg s_{k}) = \Bigl( \left( \begin{smallmatrix} 1 & s_1 \\ 0 & 1 \end{smallmatrix} \right) \vvirg \left( \begin{smallmatrix} 1 & s_{k} \\ 0 & 1 \end{smallmatrix} \right) \Bigr).
 \end{align*} 
If $s_1 + \cdots + s_k = 0$, then $h(s_1 \vvirg s_{k}) \in H$. Let $m = \max \{ \sum_{j=1}^k i_j : P_{i_1 \vvirg i_k} \neq 0\}$ be the largest ``weight'' of an element appearing in the support of $P$.
Then
\[
h(s_1 \vvirg s_{k}) P = \eta(s_1 \vvirg s_k) v_{0 \vvirg 0} + Q 
\]
where $\eta(s_1 \vvirg s_k)$ is a polynomial in $s_1 \vvirg s_k$ of degree $m$ and no constant term and $Q$ is a tensor for which the coefficient of $v_{0 \vvirg 0}$ is $0$. We will show that $\eta(s_1 \vvirg s_k)$ is not identically $0$ on the hyperplane $s_1 + \cdots + s_k = 0$. This guarantees that there is a choice of $(s_1 \vvirg s_k)$ such that $h(s_1 \vvirg s_k)$ is an element of the stabilizer subgroup $H$ of $\w_k$ and the coefficient of $v_{0 \vvirg 0}$ in $h(s_1 \vvirg s_{k})P$ is nonzero. Consider two cases:
\begin{itemize}
 \item If $ m \geq 2$, then $\eta$ is a polynomial of degree at least $2$. Note that by construction all monomials appearing in $\eta(s_1 \vvirg s_k)$ are square-free, namely no variable $s_i$ appears with exponent larger than $1$. If $\eta(s_1 \vvirg s_k) \equiv 0$ on the hyperplane $s_1 + \cdots + s_k = 0$, then the linear form $s_1 + \cdots + s_k$ must divide $\eta(s_1 \vvirg s_k)$. But every multiple of $s_1 + \cdots + s_k$ has at least one monomial which is not square-free. Therefore, $\eta(s_1 \vvirg s_k)$ does not vanish identically on the hyperplane $s_1 + \cdots + s_k = 0$.
 \item If $m =1$, then the support of $P$ is contained in the support of $\w_k$. In this case $\eta(s_1 \vvirg s_k) = P_{1, 0 \vvirg 0} s_1 + \cdots + P_{0 \vvirg 0, 1}s_k$. Since $P$ and $\w_k$ are linearly independent, the coefficients $P_{1, 0 \vvirg 0} \vvirg P_{ 0 \vvirg 0,1}$ of $P$ are not all equal. Therefore $\eta(s_1 \vvirg s_k)$ is not a scalar multiple multiple of $s_1 + \cdots + s_k$ and so it does not vanish identically on the hyperplane $s_1 + \cdots + s_k = 0$.
\end{itemize}
Define $\tilde{P} = h(s_1 \vvirg s_{k}) (P)$ for a generic choice of $s_1 \vvirg s_{k}$. The argument above shows that the coefficient of $v_{0 \vvirg 0}$ in $\tilde{P}$ is nonzero. From the first part of the proof, we obtain 
\[
v_{0 \vvirg 0} \in \bar{(\bbK^\times \times H )\cdot \tilde{P}} =  \bar{(\bbK^\times \times H )\cdot P},
\]
and this concludes the proof.
\end{proof}

We will now prove \autoref{thm: degeneration to w explicit}. Besides \autoref{prop: degen with Stabw}, this requires the results on partition rank and restriction of \autoref{sec:prank-restr} and the theory on degeneration and Grassmannians from \autoref{sec:degen-grass}.

\begin{proof}[{Proof of \autoref{thm: degeneration to w explicit}}]
By \autoref{prop: good flattenings under restriction}, after possibly applying a generic restriction to $T$, we may assume $\dim V_j = 2$ for every $j$.

The proof is by induction on the order $k$ of $T$. For the base case $k=2$, the statement is clearly true. 

For the induction step, fix $k \geq 3$ and assume the result is true for tensors of order $k-1$. Let $T \in V_1 \ootimes V_k$ be a tensor of order $k$ satisfying $\pR(T)\geq 2$. By \autoref{lemma: characterization good flattening ranks}, we can write
 \[
  T = S \otimes v^{(k)}_0 + S' \otimes v^{(k)}_1 
 \]
 where $V_k = \langle v^{(k)}_0 , v^{(k)}_1 \rangle$, $S$ is a tensor of order $k-1$ satisfying $\pR(S) \geq 2$ and $S'$ is linearly independent from $S$. Define the subspace $E_T := \langle S, S' \rangle$ which is the image of the flattening map $T : V_k^* \to V_1 \ootimes V_{k-1}$.
 
The tensor $S$ has order $k-1$ and partition rank at least two, so by the induction hypothesis, we know that $S$ degenerates to $\w_{k-1}$. By definition of algebraic degeneration, there exists an element $g_\eps \in \GL(V_1^{\ext}) \ttimes \GL(V_{k-1}^{\ext})$ such that $g_\eps S = \w_{k-1} + O(\eps)$. Consider the action of $ \GL(V_1^{\ext}) \ttimes \GL(V_{k-1}^{\ext})$ on $\Gr(2,V_1^{\ext} \ootimes V_{k-1}^{\ext})$. Regarding $g_\eps \cdot E_T$ as an element of $\Lambda^2 (V_1^{\ext} \ootimes V_{k-1}^{\ext})$, we have 
\[
g_\eps \cdot E_T =  (g_\eps S) \wedge (g_\eps S') = (\w_{k-1} +O(\eps)) \wedge (g_\eps S') .
\]
We can write $g_\eps S' = \frac{1}{\eps^a} \sum_{j=0}^\infty P_j \eps^j$ for some integer $a$ and tensors $P_j \in V_1 \ootimes V_{k-1}$. Let $j_0 = \min\{ j : P_{j} \text{ is linearly independent from } \w_{k-1}\}$; in particular $P_{j_0}\neq 0$. After possibly rescaling $g_\eps$ by a power of~$\eps$, we have
\[
g_\eps \cdot E_T = \w_{k-1} \wedge P_{j_0} + O(\eps).
\]
Write $P = P_{j_0}$ and define 
\[
\hat{T} = \w_{k-1} \otimes v^{(k)}_0 + P \otimes  v^{(k)}_1.
\]
Notice $g_\eps \cdot E_T = E_{\hat{T}} + O(\eps)$  where $E_{\hat{T}} = \langle \w_{k-1} , P\rangle$ is the flattening image $\im(\hat{T} : V_k^* \to V_1 \ootimes V_k)$ of $\hat{T}$, regarded as an element of $\Gr(2, V_1 \ootimes V_k)$. We have shown that $E_T$ degenerates to the $E_{\hat{T}}$ under the action of $\GL(V_1) \ttimes \GL(V_{k-1})$ on $\Gr(2,V_1 \ootimes V_{k-1})$. Therefore, by \autoref{thm: degenerations via grassmannians}, the tensor $T$ degenerates to the tensor $\hat{T}$ under the action of $(\GL(V_1) \ttimes \GL(V_{k-1})) \times \GL(V_k)$ on $V_1 \ootimes V_k$. Thus, to prove that $T$ degenerates to $\w_k$, it is enough to show that $\hat{T}$ degenerates to $\w_{k}$. 

To prove that $\hat{T}$ degenerates to $\w_{k}$ we use \autoref{prop: degen with Stabw}. Since $P$ is linearly independent from $\w_{k-1}$, \autoref{prop: degen with Stabw} guarantees that there exists an element $h_\eps \in (\bbK^\times \times H)^{\ext}$ where $H= \Stab_{\GL(V_1) \ttimes \GL(V_{k-1})}(\w_{k-1})$, such that $ h_\eps \cdot P = v^{(1)}_0 \ootimes v^{(k-1)}_{0} +O(\eps)$. 

Now, up to rescaling $h_\eps$, we have
\[
h_\eps \cdot  E_{\hat{T}} = (h_\eps \cdot \w_{k-1}) \wedge (h_\eps \cdot P) = \w_{k-1} \wedge v^{(1)}_0 \ootimes v^{(k-1)}_{0} +O(\eps),
\]
showing $E_{\hat{T}}$ degenerates to $\langle \w_{k-1} , v^{(1)}_0 \ootimes v^{(k-1)}_{0}\rangle = \im(\w_{k} : V_k^* \to V_1 \ootimes V_{k-1})$ under the action of $H \subseteq \GL(V_1) \ttimes \GL(V_{k-1})$. Again, using \autoref{thm: degenerations via grassmannians}, we conclude $\hat{T}$ degenerates to $\w_{k}$ under the action of $\GL(V_1)\ttimes \GL(V_{k-1}) \times \GL(V_k)$. This concludes the proof.
\end{proof}

\section{A gap in the subrank of tensors}\label{sec:gap}

In this section, we prove the following theorem and we discuss how to obtain from it all results stated in \autoref{intro}, in particular the subrank gap of \autoref{th:subrank-gap}.

\begin{theorem}\label{th:gap-flattening}
For every nonzero $T \in \bbK^{n_1} \ootimes \bbK^{n_k}$, the following are equivalent:
\begin{enumerate}[\upshape(a)]
\item $T$ has no flattening of rank one;
\item $\subrank(T^{\boxtimes N}) > 1$ for some $N$;
\item $\subrank(T^{\boxtimes N}) \geq c_k^{N - o(N)}$ for all $N$, where $c_k = k /(k-1)^{(k-1)/k}$.
\end{enumerate}
\end{theorem}

Notice that \autoref{thm: degeneration to w explicit} immediately implies \autoref{thm: degeneration to w}.
To prove \autoref{th:gap-flattening} we need three additional lemmas.
The first lemma says that the asymptotic subrank does not change under field extension; in particular, statements on the asymptotic subrank over algebraically closed fields extend to arbitrary fields.
\begin{lemma}[{\cite[Thm.~3.10]{strassen1988asymptotic}}]\label{lem:field-ext}
    Let $\bbK$ be any field and let $\bbL$ be any field that extends $\bbK$. Let $T \in \bbK^{n_1} \otimes \cdots \otimes \bbK^{n_k}$ be any tensor over $\bbK$. Let $T_{\bbL} \in \bbL^{n_1} \otimes \cdots \otimes \bbL^{n_k}$ be $T$ as a tensor over $\bbL$. Then $\asympsubrank(T) = \asympsubrank(T_\bbL)$. In other words, $\subrank(T^{\boxtimes N}) \geq c^{N - o(N)}$ if and only if $\subrank(T_\bbL^{\boxtimes N}) \geq c^{N - o(N)}$.
\end{lemma}
The second lemma describes the growth of the subrank of powers of the tensor $\w_k$, which is where the constant $c_k = 2^{h(1/k)}$ originates.
\begin{lemma}[{\cite{strassen1991degeneration,ChrVraZui:AsyRankGraph}}]\label{lem:subrank-w}
For every $k\geq 2$ we have $\subrank(\w_k^{\boxtimes N}) = c_k^{N - o(N)}$.
\end{lemma}
The third lemma guarantees that the subrank is monotone under degeneration in an asymptotic sense.
\begin{lemma}[\cite{strassen1988asymptotic}]\label{lem:degen-asymp-subrank}
For any tensors $T$ and $S$, if $T \degengeq S$ and $\subrank(S^{\boxtimes N}) \geq d^{N - o(N)}$ for some constant $d$, then $\subrank(T^{\boxtimes N}) \geq d^{N - o(N)}$.
\end{lemma}

\begin{proof}[Proof of \autoref{th:gap-flattening}]
(a) $\Rightarrow$ (c). By \autoref{lem:field-ext} we may assume that $\bbK$ is algebraically closed. If $T$ has no flattening of rank one, then \autoref{thm: degeneration to w} guarantees that $T \degengeq \w_k$. From \autoref{lem:subrank-w} and \autoref{lem:degen-asymp-subrank}, we conclude $\subrank(T^{\boxtimes N}) \geq c_k^{N - o(N)}$. 

(c) $\Rightarrow$ (b). This is clear.

(b) $\Rightarrow$ (a). If $T$ has a flattening of rank one, then $\subrank(T^{\boxtimes N}) \leq 1$ for all $N$. The contrapositive is the implication we are looking for.
\end{proof}

We now prove the results in 
\autoref{intro} as a consequence of \autoref{th:gap-flattening}.

\thmsubrankgap*
\begin{proof}
If $T$ has a flattening of rank one, then by \autoref{th:gap-flattening} $\subrank(T^{\boxtimes N}) \leq 1$ for all $N$. Since $T$ is nonzero we must have $\subrank(T^{\boxtimes N}) = 1$ for all $N$.
If $T$ has no flattening of rank one, then \autoref{th:gap-flattening} guarantees $\subrank(T^{\boxtimes N}) \geq c_k^{N - o(N)}$ for all $N$.
\end{proof}

When $k=2$, the bound of \autoref{th:subrank-gap} already holds non-asymptotically. Indeed, for $k=2$ the tensor $T$ is a matrix and the subrank $\subrank(T)$ is the matrix rank of $T$. In this case $\w_2$ is isomorphic to $\un_{2,2}$. Clearly, the rank of any matrix is at most 1 or at least $c_2 = 2$.

We will now discuss the short proofs of \autoref{th:pr-gap} and \autoref{th:f-gap}.

\thmprgap*
\begin{proof}
If $\pR(T) = 1$ then $\pR(T^{\boxtimes N}) = 1$ for every $N$. 
If $\pR(T) \geq 2$, then \autoref{th:subrank-gap} guarantees $\rmQ(T^{\boxtimes N}) \geq c_k^{ N - o(N)}$. By \autoref{lem:subrank} the subrank is a lower bound to the partition rank, so we conclude $\pR(T^{\boxtimes N}) \geq c_k^{ N - o(N)}$. 
\end{proof}

\thmgeneralgap*
\begin{proof}
The assumption is that $T$ has no flattening of rank one, therefore $\rmQ(T^{\boxtimes N}) \geq c_k^{N - o(N)}$ for every $N$. Write explicitly $\rmQ(T^{\boxtimes N}) \geq r_N = c_k^{N - \alpha_N}$ where $\frac{1}{N}\alpha_N \to 0$ as $N \to \infty$. We may assume that the $r_N$ are integers. By definition of subrank, we have $T^{\boxtimes N} \geq \un_{k,r_N}$. Therefore, for every normalized monotone, we obtain $f(T^{\boxtimes N}) \geq f( \un_{k,r_N}) = r_N = c^{N-\alpha_N}$. We conclude $f(T^{\boxtimes N}) \geq c_k^{N - o(N)}$ as desired.
\end{proof}

Finally, we provide an explicit proof of \autoref{th:asysubrank}.
\thmasymsubrank*
\begin{proof}
 By definition, $\aQ(T) = \lim_{N \to \infty} \rmQ(T^{\boxtimes N})^{1/N}$. If $\rmQ(T^{\boxtimes N}) = 1$ for every $N$, then clearly $\aQ(T) = 1$. Otherwise, applying \autoref{th:subrank-gap}, we get that $\aQ(T)  = \lim_{N \to \infty} \rmQ(T^{\boxtimes N})^{1/N} \geq \lim_{N\to \infty} (c_k^{N - o(N)})^{1/N} = c_k$ as desired.
\end{proof}

\section{A second gap in the subrank of tensors of order three}\label{sec:orderthree}

\autoref{th:subrank-gap} is a result providing a gap in the subrank of large Kronecker powers for tensors of arbitrary order. In this section we will prove a ``second gap'' for the special case of tensors $T \in \bbK^{n_1} \otimes \bbK^{n_2} \otimes \bbK^{n_3}$ of order three ($k=3$). 

\thmasympsubtrich*

To prove \autoref{th:asympsubtrich} we prove the following structural result, which strengthens \autoref{thm: degeneration to w} for the special case of tensors of order three.

\thmorderthreecomplex*

The proof of \autoref{thm: subrank 2 for 3tensors} uses the orbit classification of tensors in $\bbK^2 \otimes \bbK^2 \otimes \bbK^2$ under the action of $\GL(\bbK^2) \times \GL(\bbK^2) \times \GL(\bbK^2)$. This orbit classification is as follows:

\begin{lemma}\label{lem:classification}
For every $T \in \bbK^2 \otimes \bbK^2 \otimes \bbK^2$ exactly one of the following statements holds:
\begin{enumerate}[\upshape(a)]
    \item $T = 0$;
    \item $T \cong e_1 \otimes e_1 \otimes e_1$;
    \item $T \cong e_1 \otimes e_1 \otimes e_1 +  e_1 \otimes e_2 \otimes e_2$;
    \item $T \cong e_1 \otimes e_1 \otimes e_1 +  e_2 \otimes e_1 \otimes e_2$;
    \item $T \cong e_1 \otimes e_1 \otimes e_1 +  e_2 \otimes e_2 \otimes e_1$;
    \item $T \cong \w_3 = e_1 \otimes e_1 \otimes e_2 + e_1 \otimes e_2 \otimes e_1 + e_2 \otimes e_1 \otimes e_1$;
    \item $T \cong \un_{3,2} = e_1 \otimes e_1 \otimes e_1 + e_2 \otimes e_2 \otimes e_2$.
\end{enumerate}
\end{lemma}
Notice that $\pR(T) = 1$ in cases (b)-(e) and $\pR(T) = 2$ in cases (f),(g). When $\bbK = \bbC$ \autoref{lem:classification} follows essentially from the results of Sylvester~\cite{Sylv:PrinciplesCalculusForms}. When~$\bbK$ is an arbitrary field \autoref{lem:classification} can be obtained using Kronecker's classification of matrix pencils, for which we refer to \cite[Ch.~XII]{Gant:TheoryOfMatrices}.

To continue we will study the following set.
\begin{definition}
Let $n_1, n_2, n_3 \geq 2$. Denote by $\tau(n_1, n_2, n_3) \subseteq \bbK^{n_1} \otimes \bbK^{n_2} \otimes \bbK^{n_3}$ the Zariski closure of the orbit of $\w = e_1 \otimes e_1 \otimes e_2 + e_1 \otimes e_2 \otimes e_1 + e_2 \otimes e_1 \otimes e_1$ under the action of $\GL(\bbK^{n_1}) \times \GL(\bbK^{n_2}) \times \GL(\bbK^{n_3})$.
\end{definition}
By definition, $\tau(n_1,n_2,n_3)$ is Zariski-closed and therefore it is defined by polynomial equations; it is called the \emph{tangential variety} of the Segre variety. The study of this variety and its defining equations is the subject of a long series of works. The defining equation of $\tau(2,2,2) \subseteq \bbK^2 \otimes \bbK^2 \otimes \bbK^2$ is known since Cayley \cite{Cay:TheoryLinTransformations}; it is a polynomial equation $\Cay \in \bbK[\bbK^2 \otimes \bbK^2 \otimes \bbK^2]$ of degree four, usually called the \emph{Cayley hyperdeterminant}. We record this result in the following:
\begin{lemma}\label{lem: cayley hyperdet}
Let $a_0,a_1, b_0,b_1, c_0,c_1$ be bases of three copies of $\bbK^2$. Let $\bbK[t_{ijk}: i,j,k = 0,1] = \bbK[\bbK^2 \otimes \bbK^2 \otimes \bbK^2]$, where $t_{ijk}$ are the coordinates in the induced basis $a_i \otimes b_j \otimes c_k$. Then $\tau(2,2,2) \subseteq \bbP (\bbK^2 \otimes \bbK^2 \otimes \bbK^2)$ is the hypersurface of degree four defined by 
\begin{align*}
 \Cay = &t_{0,1,1}^2t_{1,0,0}^2-2t_{0,1,0}t_{0,1,1}t_{1,0,0}t_{1,0,1}+t_{0,1,0}^2t_{1,0,1}^2 -2t_{0,0,1}t_{0,1,1}t_{1,0,0}t_{1,1,0}\\ 
 &-2t_{0,0,1}t_{0,1,0}t_{1,0,1}t_{1,1,0}+4t_{0,0,0}t_{0,1,1}t_{1,0,1}t_{1,1,0}+t_{0,0,1}^2t_{1,1,0}^2+4t_{0,0,1}t_{0,1,0}t_{1,0,0}t_{1,1,1}\\
 &-2t_{0,0,0}t_{0,1,1}t_{1,0,0}t_{1,1,1}-2t_{0,0,0}t_{0,1,0}t_{1,0,1}t_{1,1,1}-2t_{0,0,0}t_{0,0,1}t_{1,1,0}t_{1,1,1}+t_{0,0,0}^2t_{1,1,1}^2.
\end{align*}
\end{lemma}
Over $\bbC$, a set of equations whose zero set is $\tau(n_1,n_2,n_3)$ was obtained in \cite{HolStu:HyperdeterminantalRelations}; these results were generalized for tensor products in a higher number of factors in \cite{Oed:SetTheoreticTangentialSegre}. Further, in \cite{OedRai:TangentialVarieties}, the affirmative answer to a conjecture of \cite{LanWey:TangentialVarietiesRationalHomVars} was given, providing a complete characterization of all equations vanishing on $\tau(n_1,n_2,n_3)$. 

In \autoref{lemma: set theoretic cayley module}, we give a characteristic free proof of some particular result in \cite{Oed:SetTheoreticTangentialSegre}. Set the following notation:

\begin{itemize}
 \item Let $\calJ(n_1, n_2, n_3) \subseteq \bbK[\bbK^{n_1} \otimes \bbK^{n_2} \otimes \bbK^{n_3}]$ be
the ideal generated by the polynomials $\Cay \circ (\pi_1 \otimes \pi_2 \otimes \pi_3)$ for all $\pi_i \in \End(\bbK^{n_i}, \bbK^2)$.

In other words, $\calJ(n_1, n_2, n_3)$ is the ideal generated by the Cayley hyperdeterminant composed with every restriction from $\bbK^{n_1} \otimes \bbK^{n_2} \otimes \bbK^{n_3}$ to $\bbK ^2 \otimes \bbK^2 \otimes \bbK^2$.
\item Let $\calM(n_1,n_2,n_3)$ be the ideal generated by the $3 \times 3$ minors of the flattenings in $\bbK^{n_1} \otimes \bbK^{n_2} \otimes \bbK^{n_3}$. The ideal $\calM(n_1,n_2,n_3)$ defines the \emph{subspace variety} of multilinear ranks $(2,2,2)$, in the sense of \cite[Sec. 7.1]{Lan:TensorBook}.
\end{itemize}

\autoref{lemma: set theoretic cayley module} below shows that the variety defined by $\calJ(n_1,n_2,n_3)$ is the union of $\tau(n_1,n_2,n_3)$ and the set of tensors of partition rank one. This will lead to the proof of \autoref{thm: subrank 2 for 3tensors}. First, we record a general result showing that defining equations for $\tau(n_1,n_2,n_3)$ are \emph{inherited}, in the sense of \cite{LanMan:IdealSecantVarsSegre}, from equations of $\tau(2,2,2)$; over $\bbC$, this is a consequence of the more general results of \cite{Oed:SetTheoreticTangentialSegre,OedRai:TangentialVarieties}. Given a set of polynomial equations $\calI$, write $\calV(\calI)$ for the variety that it defines.

\begin{lemma}\label{lem:M}
The variety $\tau(n_1,n_2,n_3)$ is the zero set of $\calJ(n_1,n_2,n_3) + \calM(n_1,n_2,n_3)$. In other words 
\[
 \tau(n_1, n_2, n_3) = \calV(\calJ(n_1, n_2, n_3)) \cap \calV(\calM(n_1,n_2,n_3)) .
\]
\end{lemma}
\begin{proof}
Let $\calJ = \calJ(n_1,n_2,n_3)$ and $\calM = \calM(n_1, n_2, n_3)$. We have $\calV( \calJ ) \cap \calV(\calM) = \calV(\calJ + \calM)$. The inclusion $\tau(n_1, n_2, n_3) \subseteq \calV(\calJ) \cap \calJ(\calM)$ is clearly true.

To prove the inclusion $\tau(n_1, n_2, n_3) \supseteq \calV(\calJ + \calM)$, consider $T \in \calV(\calJ+\calM)$. Then $T$ is an element of the subspace variety $\Sub_{2,2,2}(V_1 \otimes V_2 \otimes V_3)$; in particular, there exist subspaces $V_i' \subseteq V_i$ with $\dim V'_i = 2$ such that $T \in V_1' \otimes V_2' \otimes V_3'$. 
 
 Choose $\pi_i : V_i \to \bbK^2$ such that $\pi_i|_{V_i'}$ is injective. Since $T \in \calV(\calJ)$, we have that $\Cay$ vanishes at $T' = (\pi_1 \otimes \pi_2 \otimes \pi_3)(T)$ showing that $T' \in \tau(2,2,2)$.
 
Let $\iota_i : \bbK^2 \to V_i' \subseteq V_i$ be the inverse of $\pi_i|_{V'_i}$; in particular $\iota_i$ is injective. It is clear that the orbit-closure of $(\iota_1 \otimes \iota_2 \otimes \iota_3) ( \w_3)$ is $\tau(n_1,n_2,n_3)$; this shows that $(\iota_1 \otimes \iota_2 \otimes \iota_3)(\tau(2,2,2)) \subseteq \tau(n_1,n_2,n_3)$. Then
 \[
  T = (\iota_1 \otimes \iota_2 \otimes \iota_3) \circ (\pi_1 \otimes \pi_2 \otimes \pi_3) (T) \in (\iota_1 \otimes \iota_2 \otimes \iota_3) (\tau(2,2,2)) \subseteq \tau(n_1,n_2,n_3).
 \]
In other words, $T \in \tau(n_1,n_2,n_3)$ as desired.
\end{proof}

\autoref{lemma: set theoretic cayley module} characterizes the variety defined by the ideal $\calJ(n_1,n_2,n_3)$. 

\begin{lemma}\label{lemma: set theoretic cayley module}
We have 
\[
\calV(\calJ(n_1, n_2, n_3)) = \tau(n_1, n_2, n_3)  \cup P(n_1,n_2,n_3)
\]
where $P(n_1,n_2,n_3) = \{ T \in \bbK^{n_1} \otimes \bbK^{n_2} \otimes \bbK^{n_3}: \pR(T) = 1\}$ is the variety of tensors of partition rank one.
\end{lemma}
\begin{proof}
Since $P(2,2,2) \subseteq \tau(2,2,2)$, it is clear that 
\[
\tau(n_1, n_2, n_3)  \cup P(n_1,n_2,n_3) \subseteq \calV(\calJ(n_1, n_2, n_3)).
\]

To show the other inclusion, it is enough to show that if $T \in \calV(\calJ(n_1, n_2, n_3))$ and $\pR(T) \geq 2$, then $T \in \tau(n_1,n_2,n_3)$. By \autoref{lem:M}, this is equivalent to showing that if $T \in \calV(\calJ(n_1, n_2, n_3))$ and $\pR(T) \geq 2$, then $T \in \Sub_{2,2,2}$, namely the $3 \times 3$ minors of the flattenings vanish on $T$. Since every projection $\pi_j: \bbK^{n_j} \to \bbK^2$ factors through a three dimensional space $\bbK^3$, we may assume $n_1=n_2=n_3=3$.

Fix bases $a_0 \vvirg a_2$, $b_0 \vvirg b_2$, $c_0 \vvirg c_2$ of the three copies of $\bbK^3$; with abuse of notation, let $a_0,a_1$, $b_0,b_1$, $c_0,c_1$ be bases of the three copies of $\bbK^2$. We will be interested in coordinate restrictions $\pi : \bbK^3 \to \bbK^2$ mapping one of the three basis vectors to $0$ and the other two basis vectors to the two basis vectors of $\bbK^2$. Write $\pi_{A,i}$ (resp. $\pi_{B,i}$, $\pi_{C,i}$) for (any) restriction mapping $a_i$ (resp. $b_i$, $c_i$) to $0$. Let $\pi_{ijk} = \pi_{A,i} \otimes \pi_{B,i} \otimes \pi_{C,i}$.

Let $T \in \calV(\calJ(3,3,3))$ with $\pR(T) \geq 2$. Fix generic restrictions $\pi_A,\pi_B,\pi_C : \bbK^{3} \to \bbK^2$: by \autoref{prop: good flattenings under restriction}, we have $\pR((\pi_A \otimes \pi_B \otimes \pi_C)(T)) =2$, and since $T \in \calV(\calJ)$, we have $(\pi_A \otimes \pi_B \otimes \pi_C)(T) \in \tau(2,2,2)$. By Lemma \ref{lem:classification}, after possibly changing coordinates in $\bbK^2 \otimes \bbK^2 \otimes \bbK^2$, we may assume $(\pi_A \otimes \pi_B \otimes \pi_C)(T) = \w_3$. Moreover, after possibly changing coordinates in $\bbK^{3} \otimes \bbK^{3} \otimes \bbK^{3}$, we may assume that $\pi_A = \pi_{A,2}, \pi_B = \pi_{B,2}, \pi_C = \pi_{C,2}$. In other words, we can write
\[
 T = \w_3 + \sum_{\substack{i,j,k = 0 \vvirg 2 \\ \text{ at least one $2$}}} \theta_{ijk} a_i \otimes b_j \otimes c_k
\]
for some coefficients $\theta_{ijk}$. 

One can directly verify that 
\begin{align*}
\Cay( \pi_{022}(T)) &= \theta_{211}^2 , \\
\Cay( \pi_{202}(T)) &= \theta_{121}^2 , \\
\Cay( \pi_{220}(T)) &= \theta_{112}^2 ,
\end{align*}
which imply $\theta_{211} = \theta_{121} = \theta_{112} = 0$ because $T \in \calV(\calJ(3,3,3))$. Imposing these conditions, one can further verify
\begin{align*}
\Cay( \pi_{122}(T)) &= (\theta_{201} - \theta_{210})^2  , \\
\Cay( \pi_{212}(T)) &= (\theta_{021} - \theta_{120})^2 , \\
\Cay( \pi_{221}(T)) &= (\theta_{012} - \theta_{102})^2,  
\end{align*}
which imply $\theta_{201} = \theta_{210},\theta_{021} = \theta_{120},\theta_{012} = \theta_{102}$. Imposing these conditions, we obtain 
\begin{align*}
\Cay( \pi_{210} (T)) &= (\theta_{122} - \theta_{102}\theta_{120})^2,\\
\Cay( \pi_{021} (T)) &= (\theta_{212} - \theta_{012}\theta_{210})^2, \\
\Cay( \pi_{102} (T)) &= (\theta_{221} - \theta_{210}\theta_{120})^2. 
\end{align*}
This allows us to express $\theta_{122},\theta_{212},\theta_{221}$ in terms of the other coefficients. Further, we obtain 
\begin{align*}
 \Cay( \pi_{112} (T)) &= (\theta_{220} - [\theta_{120}\theta_{200}+\theta_{020}\theta_{210}])^2\\
 \Cay( \pi_{121} (T)) &= (\theta_{202} - [\theta_{102}\theta_{200}+\theta_{002}\theta_{210}])^2\\
 \Cay( \pi_{211} (T)) &= (\theta_{022} - [\theta_{020}\theta_{102}+\theta_{002}\theta_{120}])^2,
\end{align*}
which allows us to express $\theta_{220},\theta_{202},\theta_{022}$ in terms for the other coefficients. Finally, we have 
\[
 \Cay( \pi_{011} (T)) = (\theta_{222} - [\theta_{102}\theta_{220}+\theta_{002}\theta_{221}])^2.
\]
These identities allow us to express the coefficients of $T$ depending only on the six parameters $s_1 ,s_2,s_3, p_1,p_2,p_3$, as follows
\begin{align*}
 T = &a_0\otimes b_0 \otimes c_1 + a_0\otimes b_1 \otimes c_0 + a_1\otimes b_0 \otimes c_0\\  
 &+ s_1 \cdot a_2\otimes b_0 \otimes c_0 + s_2 \cdot a_0\otimes b_2 \otimes c_0 + s_3 \cdot a_0\otimes b_0 \otimes c_2  \\
     &+ p_1 (a_2 \otimes b_0 \otimes c_1 + a_2 \otimes b_1 \otimes c_0) + p_2 (a_1 \otimes b_2 \otimes c_0 + a_0 \otimes b_2 \otimes c_1)\\
     &+ p_3 (a_0 \otimes b_1 \otimes c_2 + a_1 \otimes b_0 \otimes c_2) \\
     &+(s_2 p_3 + s_3 p_2) a_0 \otimes b_2 \otimes c_2 + (s_1 p_3 + s_3 p_1) a_2 \otimes b_0 \otimes c_2 + (s_1 p_2 + s_2 p_1) a_2 \otimes b_2 \otimes c_0  \\
     &+ p_2 p_3 \cdot a_1 \otimes b_2 \otimes c_2 + p_1 p_3 \cdot a_2 \otimes b_1 \otimes c_2 + p_1 p_2 \cdot a_2 \otimes b_2 \otimes c_1  \\
     &+ (s_1p_2p_3 + p_1 s_2 p_3 + p_1p_2s_3) a_2 \otimes b_2 \otimes c_2.
\end{align*}
One can verify that the $3 \times 3$ minors of the flattenings of the tensor $T$ are identically $0$ as polynomials in $s_1 ,s_2,s_3,p_1,p_2,p_3$. This shows that $T \in \Sub_{2,2,2}$, and therefore $T \in \tau(3,3,3)$ by \autoref{lem:M}, as desired.
\end{proof}

\begin{proof}[Proof of \autoref{thm: subrank 2 for 3tensors}]
Let $T \in \bbK^{n_1} \otimes \bbK^{n_2} \otimes \bbK^{n_3}$. If $n_i \leq 1$ for any $i$, then we are in case (a).
Suppose $n_i \geq 2$ for all $i$.
Let $\pi_i : \bbK^{n_i} \to \bbK^2$ be generic linear maps. Let $T' = (\pi_1 \otimes \pi_2 \otimes \pi_3)T$.

Suppose that $\pR(T') = 1$. Then by \autoref{prop: good flattenings under restriction} $\pR(T) = 1$. In this case, $T$ is in case~(a).

Suppose that $\pR(T') = 2$. Then by the orbit-classification, see \autoref{lem:classification}, $T'$ is isomorphic to $\un_{3,2}$ or to $\w_3$. Suppose that $T'$ is isomorphic to $\un_{3,2}$. Then in particular $T \geq \un_{3,2}$, hence $T$ falls in case (c). Suppose that $T'$ is isomorphic to $\w_3$. Then by \autoref{lem: cayley hyperdet}, we have $\Cay(T') = \Cay(\w_3) = 0$. Therefore
\[
( \Cay \circ (\pi_1 \otimes \pi_2 \otimes \pi_3)) (T) = \Cay(T') = 0.
\]
Thus the Cayley hyperdeterminant vanishes on the generic restriction of $T$ to $\bbK^2 \otimes \bbK^2 \otimes \bbK^2$ and thus on all such restrictions by semicontinuity. Thus $T \in \calV(\calJ(n_1, n_2, n_3)) = \tau(n_1,n_2,n_3) \cup P(n_1,n_2,n_3)$ by \autoref{lemma: set theoretic cayley module}. Since $\pR(T') = 2$, $T \notin P(n_1,n_2,n_3)$, so $T \in \tau(n_1,n_2,n_3)$. Since $\tau(n_1,n_2,n_3)$ is the orbit-closure of $\w_3$, we conclude that $\w_3 \degengeq T$. We already knew that $T \geq \w_3$. Hence $T$ falls in case (b).
\end{proof}

\begin{proof}[Proof of \autoref{th:asympsubtrich}]
By \autoref{lem:field-ext} we may without loss of generality assume $\bbK$ is algebraically closed.

If $\pR(T) = 1$, then $\rmQ(T^{\boxtimes N}) = 1$ for every $N$. 

If $T$ is isomorphic to $\w_3$, then $\rmQ(T^{\boxtimes N}) = c_3^{N - o(N)}$ by \autoref{lem:subrank-w}. 

If $\pR(T) \geq 2$ and $T$ is not isomorphic to $\w_3$, then \autoref{thm: subrank 2 for 3tensors} guarantees that $T \geq \un_{3,2}$, therefore $\rmQ(T) \geq 2$ and $\rmQ(T^{\boxtimes N}) \geq 2^N$ for every $N$.
\end{proof}

The proof of \autoref{thm: asy rank 3 tensors} follows immediately.

\section{Open problems}\label{sec:open_problems}

Our results naturally lead to several open problems which we briefly discuss in this section.

\begin{enumerate}
    \item For every $k \geq 2$, \autoref{th:asysubrank} shows that the smallest possible values for the asymptotic subrank~$\aQ(T) = \lim_{N \to \infty} \subrank(T^{\boxtimes N})^{1/N}$ of any $k$-tensor $T$ are the values
    \begin{align*}
    q_0&= 0\\
    q_1&= 1\\
    q_2&= c_k
    \end{align*}
    where $c_k = k/(k-1)^{k/(k-1)} = 2^{h(1/k)}$ with $h$ denoting the binary entropy function.
    Is there again a gap between $q_2$ and the next possible value of the asymptotic subrank? The answer is known to be ``yes'' over finite fields \cite[Corollary~1.4.3]{BlaDraRup:TensorRestrictionFiniteFields}. Or is $q_2$ an accumulation point for the possible values of the asymptotic subrank? 
    \item It is natural to specialize the previous question to small values of $k$. For tensors of order $k=2$ (matrices) the situation is completely understood, and the possible values of the asymptotic subrank are precisely the natural numbers. For tensors of order $k = 3$ the situation is already much more complicated and the answer to the above questions is not known. \autoref{thm: asy rank 3 tensors} show that the smallest possible values of the asymptotic subrank are 
    \begin{align*}
        q_0 &= 0\\
        q_1 &= 1\\ 
        q_2 &= c_3 \mathrlap{ {}\approx 1.88988}\\
        q_3 &= 2.
    \end{align*} 
    What is the next possible value? It is known that there exists a tensor $T$ of order three with $\aQ(T) \approx 2.68664$\cite[page~169]{strassen1991degeneration} and that there exists a tensor $T$ of order three with $\aQ(T) \approx 2.7551$ \cite[page~132]{strassen1991degeneration}. 
    \item As the main ingredient for the proof of \autoref{th:subrank-gap} we prove that for any $k$-tensor $T$, no flattening of $T$ has rank one if and only if $T \degengeq \w_k$. We may phrase this as a ``degeneration duality'' as follows. For any tensor $T$, there is no degeneration $S \degengeq T$ for any partition rank one tensor $S$ if and only if $T \degengeq \w_k$. In other words, the existence of certain degenerations corresponds to the non-existence of other degenerations. What other degeneration dualities for tensors exist?
\end{enumerate}

\begin{remark}
     Progress on these open problems was achieved after the submission of this work. In \cite{BrChrLeiShpZu:Discreteness} the authors proved discreteness of several asymptotic tensor parameters for 3-tensors in several regimes. In particular, they proved this for the asymptotic subrank and asymptotic slice rank over finite fields, and for asymptotic slice rank over the field of complex numbers.

     Finally, \cite{GesZui:NextGap} determined the next element in the set of values of asymptotic subrank. If $T$ is a $3$-tensor with $\aQ(T) > 2$, then $\aQ(T) \geq q_4$, where $q_4 \approx  2.68664$ is the asymptotic subrank of the structure tensor of the null-algebra of dimension $3$ \cite{strassen1991degeneration}.
     \end{remark}

\paragraph{Acknowledgements.}
 M.C.\ acknowledges financial support from the European Research Council (ERC Grant Agreement No.\ 81876), VILLUM FONDEN via the QMATH Centre of Excellence (Grant No.\ 10059) and the Novo Nordisk Foundation (grant NNF20OC0059939 ``Quantum for Life''). J.Z.\ was supported by a Simons Junior Fellowship and NWO Veni grant VI.Veni.212.284. This work is partially supported by  the Thematic Research Programme ``Tensors: geometry, complexity and quantum entanglement'', University of Warsaw, Excellence Initiative - Research University and the Simons Foundation Award No.\ 663281 granted to the Institute of Mathematics of the Polish Academy of Sciences for the years 2021--2023. We thank D.\ Agostini and J.\ Jelisiejew for helpful discussions regarding the geometry of curves in Grassmannians and on the equivalence between algebraic and topological degeneration.

{\small
\bibliographystyle{alphaurl}
\bibliography{subrank}
}

\end{document}